\documentclass[10pt]{amsart}
\usepackage{amsmath}
\usepackage{amssymb}
\usepackage{amsthm}
\usepackage{amsfonts, dsfont}
\usepackage{paralist}
\usepackage{graphics} 
\usepackage{epsfig} 
\usepackage{graphicx}  
\usepackage{epstopdf}
\usepackage{epstopdf}
\usepackage{verbatim}
\usepackage{mathrsfs}
\usepackage{mathtools}
\usepackage{pstricks}

\usepackage{relsize}
\usepackage{tikz}
\usetikzlibrary{matrix}
\usepackage{subcaption}
\usepackage{pgfplots}
\usepackage{fixltx2e}
\usepackage{enumitem}
\usepackage{ upgreek }
\usepackage[colorlinks=true]{hyperref}
\hypersetup{urlcolor=blue, linkcolor=blue, citecolor=red}
\usepackage{hyperref}
\usepackage{cleveref}
\usepackage{bm}
\usepackage{appendix}
\newtheorem{theorem}{Theorem}[section]
\newtheorem{lemma}[theorem]{Lemma}
\newtheorem{corollary}[theorem]{Corollary}
\newtheorem{proposition}[theorem]{Proposition}
\newtheorem{remark}[theorem]{Remark}
\newtheorem{definition}[theorem]{Definition}

\numberwithin{equation}{section}

\DeclareMathOperator*{\argmin}{argmin}

\begin{document}

\title[Viscous Ergodic Problem. Part 1: HJB Equation]{A viscous ergodic problem with unbounded and measurable ingredients. Part 1: HJB Equation}
\thanks{The author is funded by the Deutsche Forschungsgemeinschaft (DFG, German Research Foundation) – Projektnummer 320021702/GRK2326 – Energy, Entropy, and Dissipative Dynamics (EDDy). An earlier, yet incomplete, version of this manuscript was part the author's Ph.D. thesis \cite{kouhkouhPhD} which was conducted when he was a Ph.D. student at the University of Padova.}

\author{Hicham Kouhkouh}
\address{Hicham Kouhkouh \newline \indent
{RWTH Aachen University, Institut f\"ur Mathematik,  \newline \indent 
RTG Energy, Entropy, and Dissipative Dynamics,\newline \indent
Templergraben 55 (111810)},
 \newline \indent 
{52062, Aachen, Germany}
}
\email{\texttt{kouhkouh@eddy.rwth-aachen.de}}

\date{\today}
%
\begin{abstract}
We address the problem of existence and uniqueness of solutions $(c,u(\cdot))$ to ergodic Hamilton-Jacobi-Bellman (HJB) equations of the form $H(x,\nabla u(x), D^{2}u(x)) = c$
in the whole space $\mathds{R}^{m}$ with unbounded and merely measurable data and where $H$ is a Bellman Hamiltonian. 
The method we use is different from classical approaches. It relies on duality theory and optimization in abstract Banach spaces together with maximal dissipativity of the diffusion operator.
\end{abstract}

\maketitle


\section{Introduction}

This paper is devoted to the problem of existence  of solutions to some ergodic fully nonlinear partial differential equations in the whole space domain $\mathds{R}^{m}$ with unbounded and measurable data satisfying a subexponential growth. 
Such a problem takes the form of
\begin{equation} \label{eq: intro - pde}
    \textit{Find $(c,u(\cdot))\in\mathds{R}\times \mathcal{X}(\mathds{R}^{m})$ s.t.: } \; H(x,\nabla u(x),D^{2}u(x)) = c,\;\text{ in }\mathds{R}^{m}
\end{equation}
where $\mathcal{X}$ is a functional space (part of the unknowns), $H$ is a Bellman Hamiltonian 
\begin{equation*}
        H(x,\nabla u(x),D^{2}u(x)) \coloneqq \min\limits_{\alpha\in A}\{\,-\mathcal{L}_{\alpha}u(x) + f(x,\alpha)\,\}
    \end{equation*}
and $\mathcal{L}_{\alpha}$ is a diffusion operator
\begin{equation*}
    \mathcal{L}_{\alpha}\varphi(x) :=  \text{trace}(a(x,\alpha)D^{2}\varphi(x)) + b(x,\alpha)\cdot\nabla\varphi(x)
\end{equation*}
with $\alpha\in A$ a compact subset of $\mathds{R}^{k}$ for some $k>0$. The case where $H$ is given with a $\max$ (instead of a $\min$) can be obtained analogously (see \cite{kouhkouhPhD} for further details).

This problem arises in ergodic stochastic control, weak KAM theory, homogenization, singular perturbations and asymptotic approximations in partial differential equations (long-time behavior, vanishing discount coefficient). It has also been shown recently in \cite{bardi2023eikonal} that it is linked to global optimization.

Throughout this paper, we will make the following assumptions and refer to them wherever it is needed:

\begin{description}
    \item[A1]{
\begin{enumerate}[label = (\roman*)]
        \item $a=(a^{ij}_{\alpha})$ is a Lipschitz continuous mapping (uniformly in $\alpha$) on $\mathds{R}^m$ such that $a(x,\alpha)=\varrho(x,\alpha)\varrho(x,\alpha)^{\top}$ where $\varrho$ is a continuous in $x$ (unif. in $\alpha$) $m\times m_{1}$ matrix function (for some $m_{1}\geq m$), 
        \item $b=(b^{i}_{\alpha}):\mathds{R}^m\times A\to \mathds{R}^m$ is a locally bounded Borel-meas. vector field.
    \end{enumerate}}
    \item[A2]{For $p>m$, $a^{ij}(\cdot,\alpha)\in W^{p,1}_{\text{loc}}(\mathds{R}^{m})$ and $b^{i}(\cdot,\alpha)\in L^{p}_{\text{loc}}(\mathds{R}^{m})$, uniformly in $\alpha\in A$.   }
    \item[A3]{There exist $\overline{\Lambda}\,\geq\,\underline{\Lambda}\;>0$ such that $\forall\;x,\xi\in\mathds{R}^m$, $\;\underline{\Lambda} \|\xi\|^{2}\,\leq\, \xi \cdot a(x,\alpha) \xi \,\leq\, \overline{\Lambda}\|\xi\|^{2}$, \\ 
    uniformly in $\alpha\in A$, i.e. $(a^{ij})$ is positive, unif. bounded and nondegenerate.}
    \item[A4]{The drift $b$ satisfies, for some positive numbers $\chi, \gamma_{1},\gamma_{2}$,
    \begin{equation*}
        \sup\limits_{\alpha\in A} \, b(x,\alpha)\cdot x  \leq \gamma_{1} - \gamma_{2}|x|^{\chi},\quad \forall\, x\in\mathds{R}^{m}.
    \end{equation*}} 
    \item[A5]{$x\mapsto f(x,\alpha)$ is  Borel-measurable on $\mathds{R}^{m}$ with at most a polynomial growth,  i.e. 
    $\exists\;K_{f}>0,\; \text{s.t. } |f(x,\alpha)|\leq K_{f}(1+|x|)^{d}$, $\forall\;x\in\mathds{R}^{m},\alpha\in A$ and for some $d\geq 1$.}
    \item[A6]{$\exists \, K_{b}>0$ and $\theta\in [0,d]$ such that $|b(x,\alpha)|\leq K_{b}(1+|x|)^{\theta}$ for all $x\in\mathds{R}^{m},\alpha\in A$.}
\end{description}


We will also need some assumptions on how the data of the problem depend on the parameter $\alpha$. These shall be set later in \S \ref{sec: prelim}. 


The main difficulty and novelty in this setting is that we are looking for solutions in the whole space $\mathds{R}^{m}$ while both $b$ and $f$ are unbounded. Usually, we refer to $c$ as the \textit{ergodic constant} (or, critical constant) and $u(\cdot)$ as the \textit{corrector} (or, critical solution). The differential operator $\mathcal{L}_{\alpha}$ can be interpreted as the infinitesimal generator of the controlled stochastic process
\begin{equation}\label{eq: SDE intro}
    dX_{t} = b(X_{t},\alpha_{t})dt + \sqrt{2}\varrho(X_{t},\alpha_{t})dB_{t}
\end{equation}
where $B_{t}$ is a Wiener process while $f$ is the running cost of the control problem. 
Note that \eqref{eq: SDE intro} should be understood in its weak sense (see e.g. \cite{krylov1969ito, krylov37selection, lee2022analytic}).


The \textbf{main result} (see  
Theorem \ref{thm: main NL}) can be informally stated as follows:
\textit{
Under assumptions including (A1)-(A6), the following statements hold true:
\begin{enumerate}[label = (\roman*)]
    \item (Existence)\; There exists a  constant $c\in\mathds{R}$ such that the PDE in \eqref{eq: intro - pde} admits an almost everywhere solution $u(\cdot)\in W^{r,2}_{\text{loc}}(\mathds{R}^{m})$ with $r\in [1,+\infty)$ and satisfying $|u(x)|\leq K(1+|x|^{\kappa})$ where $K>0$ and $\kappa = d+1-\theta$.
    \item (Uniqueness)\;If, for the specific constant $c$ shown in (i), we assume moreover that $b$ is locally Lipschitz continuous with at most a linear growth (i.e. $\beta=1$ in (A6)), then $u(\cdot)$ is unique in $W^{r,2}_{\text{loc}}(\mathds{R}^{m})$ with $r>\frac{m}{2}$, up to an additive constant. That is, if $(c,u(\cdot))$ and $(c,v(\cdot))$ are two solutions in the sense of (i), then $u(\cdot)-v(\cdot)$ is a constant.
\end{enumerate}
}


In particular, the constant $c$ is shown to be the \textit{critical} (largest) one with an explicit formula allowing us to derive a continuity estimate with respect to the data of the problem. 


\textbf{Related results. } The ergodic problem captures the asymptotic behavior of a system (e.g. the long-time behavior of a control problem, or the effective phenomena in homogenization) and hence plays the role of a model reduction technique that is of interest in many applications. In the context of stochastic control, such a problem arises for the first time in the pioneering work \cite{lasry1975controle}.
Then probably the first result linking homogenization to ergodic theory goes back to \cite{bensoussan1979homogenization}, and the ergodic problem as we have stated appears in the context of homogenization in \cite{lions1986homogenization}. Since then many results on the problem and related topics have been established.\\
\textbullet\quad \textbf{In the linear case. } This corresponds to the \textit{ergodic Poisson equation}, that is to find a pair $(c,\varphi(\cdot))$ where $c$ is a constant, that solves $c+\mathcal{L}\varphi=g$ in the whole space. If one already knows what a possible ergodic constant $c$ can be, then this boils down to the usual Poisson equation $\mathcal{L}\varphi = \widetilde{g}$ where $\widetilde{g}\coloneqq g-c$. In this case, the methods used are mainly of stochastic analysis (Feynman-Kac representation), Dirichlet forms and semigroups \cite{lorenzi2006analytical}. With assumptions similar to ours, the problem $\mathcal{L}\varphi = \widetilde{g}$ is solved in \cite{pardoux2001poisson} (see also \cite{rockner2019strong,rockner2021diffusion}) under the additional assumption $\int \widetilde{g}d\mu = 0$ where $\mu$ is the invariant measure associated to $\mathcal{L}$. In fact, with our result, we get $c=\int g d\mu$ and hence $c+\mathcal{L}\varphi = g$ becomes $\mathcal{L}\varphi = g - c = \widetilde{g}$ and our problem falls in the setting of \cite{pardoux2001poisson}. In our previous work \cite{bardi2023singular} (see also \cite{kouhkouhPhD}), we constructed the ergodic constant $c$ and showed that it corresponds indeed to the mean of $g$ w.r.t. $\mu$  using probabilistic techniques and without the need of proving the existence of the solution $\varphi(\cdot)$.  We also mention \cite{mannucci2016ergodic} where linear subelliptic operators are considered in the whole  $\mathds{R}^{m}$ with possibly unbounded coefficients. The methods used in the latter are inspired by \cite{lions2015equations}.\\
\textbullet\quad \textbf{In the nonlinear case. } 
            Most of the theory has been developed for the multidimensional torus where one enjoys compactness. In such a setting the problem is treated in the seminal work \cite{arisawa1998ergodicsto}. There have been since then a wide literature, mainly in the context of long-time behavior of HJB equation and of homogenization, both for the first-order and second-order equations, but also in the context of weak KAM theory: we do not review it here, since it does not address the problem studied in this paper. The  recent work in \cite{bianca2018new} uses the link between PDEs and dynamical systems. And probably the first results treating the second-order ergodic Bellman equation on the whole space $\mathds{R}^{m}$ are \cite{bensoussan1985bellman}, then \cite{bensoussan1992bellman}.  In \cite{barles2016unbounded} (see also \cite{barles2020large}), the ergodic problem considered is of the form
\begin{equation}
\label{eq: hjb gamma}
	- \Delta u + \frac{1}{\gamma}|Du|^{\gamma} = f(x) + c,\quad \text{ in } \mathds{R}^{m}.
\end{equation}
Classical solutions are shown to exist using PDE methods, assuming some growth and regularity conditions. These results are similar to those previously shown in \cite{ichihara2011recurrence,ichihara2012large,ichihara2013criticality,ichihara2015generalized}   using methods of stochastic control theory and probability tools. Similar arguments are used in \cite{kaise2006structure} for quadratic Hamiltonian arising in risk-sensitive stochastic control problems. A study of the underlying (controlled) stochastic process can also be useful to derive helpful ergodic properties which then yield some compactness. This is done for example in the recent paper \cite{chasseigne2019ergodic} where an inward drift is assumed (similar to our assumption (A4)). Another approach that uses the stochastic ergodic control formulation together with PDE methods is the one in \cite{cirant2014solvability} where the problem considered is of the form \eqref{eq: hjb gamma} with an additional term of the form $-b(x)\cdot Du$, and the term $\frac{1}{\gamma}|Du|^{\gamma}$ is replaced by $H(Du)$ with $H$ satisfying some regularity and growth assumptions.  In the latter, the problem is approximated by a sequence of truncated problems (bounded with Neumann condition) as in \cite{lions2002ergodicity}. The usual PDE method for dealing with the viscous ergodic HJB equation as being a limiting problem of either the long-time behavior of parabolic equations or to vanishing-discount coefficient in elliptic equations is described in detail in \cite{alvarez2010ergodicity}. On the other hand, \cite{arapostathis2019uniqueness} is devoted to uniqueness of classical solutions to HJB equation of the form \eqref{eq: hjb gamma} in the case where $\gamma \in (1,2)$, and it relies on an infinite dimensional linear program for elliptic equations for measures which is an approach that is reminiscent of ours.

Our method relies on duality tools together with the extension of the diffusion operator $\mathcal{L}$. The idea is to isolate the two terms $c$ and $f$ and consider them as (part of) objective functions in suitable optimization problems dual to each other. Then we interpret a solution $(c,u(\cdot))$ of \eqref{eq: intro - pde} as dual variables of an optimization problem over the space of measures $\mu$ and whose admissible set is made of measures solving $\mathcal{L}^{*}\mu = 0$. And provided we can solve the latter equation, which is in fact a stationary \textit{Fokker-Planck-Kolmogorov} (FPK) equation, we can describe the admissible set of the optimization problem and hence recover existence of its corresponding dual variables (analogous to Lagrange multipliers). In fact, this method allows us to transpose to problems of the form \eqref{eq: intro - pde} the information one can get from the study of the operator $\mathcal{L}$ and its adjoint $\mathcal{L}^{*}$ through a duality scheme for suitably chosen optimization problems.

This optimization view point is not totally new since it is briefly mentioned in \cite[\S 6.6]{arapostathis2012ergodic} and is  also reminiscent of \cite{evans2002linear}. However, to our knowledge, this analysis has never been used to address a PDE problem such as the solvability of an ergodic HJB equation in our setting. Another interesting direction is the one considered in \cite{arapostathis2019uniqueness} where the problem of uniqueness of solutions to viscous HJB is addressed via similar duality methods, unlike in this paper where we use duality to prove existence only and rely rather on Liouville type results \cite{bardi2016liouville} to prove uniqueness. We would like also to mention that our method allows to deal with the ergodic HJB equation under weak regularity assumptions, in particular the dependency on the space variable is assumed to be measurable only, with a subexponential growth. Moreover our assumptions concern the coefficients of the diffusion operator (or the underlying stochastic differential equation) which is a way of presentation that is more suitable for applications in stochastic control and is different from the classical references that rather rely on structural assumptions on the Hamiltonian. Finally, the method can be extended to deal with ergodic Mean-Field Games in the same setting, which is the object of the companion manuscript \cite{kouhkouh2}.

This paper is organized as follows. In \textbf{Section \ref{sec: survey}} we provide the main results from duality theory and also from diffusion operators, in particular we define the closed extension of an operator and which is the definition we shall consider for $\mathcal{L}_{\alpha}$ in the equation \eqref{eq: intro - pde}. \textbf{Section \ref{sec: prelim}} contains some preliminary results needed in the sequel. In \textbf{Section \ref{sec: main NL}} we study the primal and dual problems, then state and prove the main result, that is the solvability of the ergodic HJB equation.

\section{Survey of known results}\label{sec: survey}

\subsection{Convex duality}
\label{sec:duality theory}

The results and remarks mentioned in this section are wellknown and can be found in \cite{bonnans2013perturbation}. For the sake of a broad readability of this paper and its self-containedness, we include the results we will use, borrowed from \cite{bonnans2013perturbation}.

Let $(X,X^{*})$ and $(Y,Y^{*})$ be paired spaces, i.e. such that each space of a pair is a locally convex topological vector space and is the topological dual of the other. We assume moreover that $X$ and $Y$ are Banach spaces that we endow with their respective strong topologies, while $X^{*}$ and $Y^{*}$ are endowed with the respective weak-$*$ topologies.

Let $Q$ be a closed convex subset of $X$ and $K$ a closed convex cone subset of $Y$. We are interested in first order optimality conditions for the optimization problem
\begin{equation}
\label{Primal - 1}
\tag{$P$}
    val(P) =\; \min\limits_{x\in Q} f(x),\quad \text{s.t.:}\;\;\; G(x)\in K
\end{equation}
where $f:X\to \mathds{R}$ and $G:X\to Y$. The objective function in \eqref{Primal - 1} can be reformulated as $f(x) + I_{Q}(x)$ while we minimize over the whole set $X$. We denote by $I_{Q}(\cdot)$ the indicator function ($I_{Q}(x) = 0$ if $x\in Q$, and $+\infty$ if $x\notin Q$). The Lagrangian of \eqref{Primal - 1} is 
\begin{equation}
\label{lagrangian}
    L(x,y^{*}) := f(x) + \langle y^{*}, G(x) \rangle,\quad (x,y^{*}) \in X\times Y^{*}.
\end{equation}

We embed the problem \eqref{Primal - 1} into the family of optimization problems
\begin{equation}
\label{Primal - y}
\tag{$P_{y}$}
    \min\limits_{x\in Q} f(x),\quad \text{s.t.:}\;\;\; G(x)+y\in K
\end{equation}
where $y\in Y$ is viewed as a parameter vector. Clearly for $y=0$, the corresponding problem $(P_{0})$ coincides with the problem \eqref{Primal - 1}. Let $v(y)$ be the corresponding value function
\begin{equation*}
    v(y) = val(P_{y}) = \inf\limits_{x\in Q}\; f(x) + I_{K}(G(x)+y). 
\end{equation*}

The (conjugate) dual of \eqref{Primal - 1} can be written in the form (see \cite[\S 2.5.3, p. 107]{bonnans2013perturbation}):
\begin{equation}
    \label{Dual}
    \tag{$D$}
    val(D) =\; \max\limits_{y^{*}\in Y^{*}} \big\{\inf\limits_{x\in Q} \;L(x,y^{*})\; - I^{*}_{K}(y^{*})\;  \big\}
\end{equation}
and $I^{*}_{K}(\cdot)$ is the Legendre-Fenchel conjugate of the indicator function supported on $K$, which is known as the \textit{support function} of the set $K$. 

Recall that $val(P)\geq val(D)$ (this can be easily obtained for example as a consequence of conjugate duality; see \cite[eq. (2.268), p.  96]{bonnans2013perturbation}, or by Lagrange duality; see \cite[Proposition 2.156, p.  104]{bonnans2013perturbation}) and that if for some $x_{o}\in Q$, $y_{o}^{*}\in Y^{*}$ the equality of primal and dual objective functions holds, i.e.
\begin{equation}
    \label{equality objective functions}
    f(x_{o}) + I_{K}(G(x_{o})) = \inf\limits_{x\in Q}\;L(x,y^{*}_{o}) - I^{*}_{K}(y_{o}^{*}),
\end{equation}
then $val(P)=val(D)$. If moreover the common value is finite, then $x_{o}\in Q$ and $y_{o}^{*}\in Y^{*}$ are optimal solutions of \eqref{Primal - 1} and \eqref{Dual} respectively. The equality \eqref{equality objective functions} can be written in the following equivalent form
\begin{equation}
    \label{equality objective functions - FY}
    \big(L(x_{o},y_{o}^{*})-\inf\limits_{x\in Q}L(x,y_{o}^{*})\big) + \big(I_{K}(G(x_{o})) + I^{*}_{K}(y_{o}^{*})-\langle y_{o}^{*},G(x_{o})\rangle\big) = 0.
\end{equation}
Clearly, the first term in the left hand side is non-negative and the second term is also non-negative by the Young-Fenchel inequality. Moreover the equality
\begin{equation*}
    I_{K}(G(x_{o})) + I^{*}_{K}(y_{o}^{*})-\langle y_{o}^{*},G(x_{o})\rangle = 0
\end{equation*}
holds if and only if $y_{o}^{*}\in \partial I_{K}(G(x_{o}))$; the subdifferential of the indicator function evaluated in $G(x_{o})$. 
Thus, the equality in \eqref{equality objective functions} is equivalent to
\begin{equation}
\label{optimality conditions - 1}
    x_{o}\in \argmin\limits_{x\in Q}L(x,y_{o}^{*})\quad \text{and}\quad y_{o}^{*}\in\partial I_{K}(G(x_{o})).
\end{equation}
Noe that $\partial I_{K}(G(x_{o}))=N_{K}(G(x_{o}))$ where $N_{K}(\cdot)$ is the normal cone\footnote{If $S\subset X$ convex, then $N_{S}(x):=\{x^{*}\in X^{*}\,: \langle x^{*},z-x\rangle \leq 0\,\forall\;z\in S\}$. If $x\notin S$ then $N_{S}(x)=\emptyset$.} to $K$. Moreover, since $K$ is a convex cone, the condition $y_{o}^{*}\in N_{K}(G(x_{o}))$ is equivalent to
\begin{equation}
    \label{equiv cond normal cone}
    G(x_{o})\in K,\quad y_{o}^{*}\in K^{-}\quad \text{and}\quad \langle y_{o}^{*},G(x_{o}) \rangle = 0
\end{equation}
where $K^{-}$ is the polar (negative dual) cone\footnote{Let $C$ be a subset of $X$, then $C^{-}:=\{x^{*}\in X^{*}\;:\; \langle x^{*},x \rangle \leq 0,\quad\forall\;x\in C \}$.} of $K$. The optimality conditions can therefore be written as
\begin{equation}
\label{optimality conditions - 2}
	x_{o}\in \argmin\limits_{x\in Q}L(x,y_{o}^{*}),\quad G(x_{o})\in K,\quad y_{o}^{*}\in K^{-}\quad \text{and}\quad \langle y_{o}^{*},G(x_{o})=0.
\end{equation}

We are interested in existence of dual variables and in \textit{no duality gap} between \eqref{Primal - 1} and \eqref{Dual}, i.e. $val\eqref{Primal - 1}=val\eqref{Dual}$. We consider the convex case which we now define before stating the existence theorem.

\begin{definition}(\cite[Definition 2.163, p. 110]{bonnans2013perturbation})\label{def: convex}
We say that the problem \eqref{Primal - 1} is convex if the function $f(x)$ is convex, the set $Q$ is convex, the set $K$ is convex and closed, and the mapping $G(x)$ is convex with respect to the set\footnote{The mapping $G$ is convex w.r.t. the set $C$ if the multifunction $G(x)+C$ is convex (see \cite[Definition 2.103, p.72]{bonnans2013perturbation}), that is, for any $x_{1},x_{2}\in X$ and $t\in [0,1]$,
\begin{equation*}
    tG(x_{1}) + (1-t)G(x_{2}) - G(tx_{1} + (1-t)x_{2}) + C \subset C.
\end{equation*}
} $C:=-K$.
\end{definition}

\begin{theorem}(\cite[Theorem 2.165, p.112]{bonnans2013perturbation})\label{thm: duality}
Consider the optimization problem \eqref{Primal - 1}. Suppose $f(x)$ is convex and lower semi-continuous, $Q$ is convex and close, $G(x)$ is continuous and \eqref{Primal - 1} is convex and satisfies
\begin{equation}
    \label{eq: interior}
    0\in \text{\upshape int}\{G(Q)-K\}.
\end{equation}
Then there is no duality gap between \eqref{Primal - 1} and \eqref{Dual}. Moreover, if $val\eqref{Primal - 1}$  is finite, then the optimal solution set of the dual problem \eqref{Dual} is a nonempty, convex, bounded and weak-$*$ compact subset of $Y^{*}$.
\end{theorem} 

In \cite{bonnans2013perturbation}, the latter definition and theorem are given for an optimization problem of the form \eqref{Primal - 1} but with $Q=X$. In our case, it is enough to consider as an objective function $f(x)+I_{Q}(x)$, instead of $f(x)$. Then, in order to have $f(x)+I_{Q}(x)$ convex (resp. lower semi-continuous), we need $f(x)$ to be a convex (resp. l.s.c.) function and $Q$ to be a convex (resp. closed) set; see \cite[Example 2.115, p.80]{bonnans2013perturbation}). Moreover, the condition \eqref{eq: interior} is written as $0\in \text{int}\{G(\text{dom}(f)) - K\}$ where $\text{dom}(f)=\{x\in X \,:\, f(x) <+\infty\}$ is the domain of $f$. In our setting, with the objective function $f(x)+I_{Q}(x)$, we have $\text{dom}(f+I_{Q}) = \text{dom}(f)\cap Q = Q$ assuming $\text{dom}(f)=X$ and $f(x)>-\infty$ for all $x\in X$.

The next proposition characterizes \eqref{eq: interior} in a particular case: \\
Let $Y$ be the Cartesian product of two Banach spaces $Y_{1}$ and $Y_{2}$, and $K=K_{1}\times K_{2}\subset Y_{1}\times Y_{2}$ where $K{1}$ and $K_{2}$ are closed convex subsets of $Y_{1}$ and $Y_{2}$ respectively. Let $G(x)=(G_{1}(x), G_{2}(x))$ with $G_{i}(x)\in Y_{i},i=1,2$. 
\begin{proposition}\label{prop: interior}
If $Y_{2}=X$, $G_{2}(x)=x$ for all $x\in X$ and $G_{1}(x)$ is $(-K)$-convex and continuously differentiable, then the following condition is equivalent to \eqref{eq: interior}
\begin{equation}
    \label{eq: interior 2}
    0\in \text{\upshape int}\{G_{1}(x_{\circ}) + DG_{1}(x_{\circ})[K_{2} - x_{\circ}]  - K_{1}\}
\end{equation}
at every feasible point $x_{\circ}\in \{x\in X\,:\, x\in Q \text{ and } G(x) \in K\}$.
\end{proposition}

\begin{proof}
Using \cite[Proposition 2.104, p.73]{bonnans2013perturbation}), we have \eqref{eq: interior} is equivalent to Robinson's constraint qualification
\begin{equation*}
    0 \in \text{\upshape int}\{G(x_{\circ}) + DG(x_{\circ})(Q-x_{\circ})  -K\}
\end{equation*}
which in turn is equivalent to \eqref{eq: interior 2}; see \cite[equation (2.192), p.71]{bonnans2013perturbation}).
\end{proof}

The following theorem concludes this subsection by characterizing the optimal solutions as desired. 

\begin{theorem}(\cite[Theorem 2.158, p.109]{bonnans2013perturbation})\label{thm: opt cond}
If val\eqref{Primal - 1}=val\eqref{Dual}, and $x_{\circ}\in X$ and $\Bar{y}^{*}\in Y^{*}$ are optimal solutions of \eqref{Primal - 1} and \eqref{Dual}, respectively, then the optimality conditions \eqref{optimality conditions - 1} hold. \\
In our setting, and as explained above, \eqref{optimality conditions - 1} are equivalently expressed by \eqref{optimality conditions - 2}.
\end{theorem}

\subsection{Extension of diffusion operators} \label{sec:extension diff}
We resume in this subsection some known results from \cite{bogachev2002uniqueness} (see also \cite{bogachev2015fokker, bogachev1999uniqueness, stannat1999nonsymmetric}). We shall be interested in a matrix-valued function $a=(a^{ij}_{\alpha})$ and a vector field $b=(b^{i}_{\alpha})$ such that $a^{ij}_{\alpha}(x) = a^{ij}(x,\alpha)$ and $b^{i}_{\alpha}(x) = b^{i}(x,\alpha)$ where $\alpha$ is some parameter in the compact set $A$. For the sake of simplicity of notations, we omit the dependence of $a$ and $b$ on the parameter $\alpha$, the latter being assumed fixed in the present subsection (its effect will be discussed next, in subsection \S \ref{sec:distance inv}). Hence we simply write $a=(a^{ij})$ a continuous mapping on $\mathds{R}^m$ and  $b=(b^i):\mathds{R}^m\to \mathds{R}^m$ a Borel-measurable vector field. Let us also set
\begin{equation}
\label{gen_diffusion op}
    L_{a,b}\varphi = a^{ij}\partial_{i}\partial_{j}\varphi + b^{i}\partial_{i}\varphi,\quad \varphi \in C^{\infty}_{0}(\mathds{R}^m),
\end{equation}
where we use the standard summation rule for repeated indices. Suppose $\mu$ is a locally finite (not necessarily non-negative) Borel measure on $\mathds{R}^m$, i.e. a measure on the Borel $\sigma$-algebra $\mathcal{B}(\mathds{R}^m)$ of $\mathds{R}^m$, solving the Fokker-Planck-Kolmogorov (FPK) equation
\begin{equation}
    \label{equation mu_diff op}
    L^{*}_{a,b}\mu = 0
\end{equation}
in the following sense:
\begin{equation}
\label{sense equation mu_diff op}
    a^{ij},b^{i}\in L^{1}_{\text{loc}}(\mathds{R}^{m};\mu) \quad \text{ and } \quad \int_{\mathds{R}^m}L_{a,b}\varphi\; d\mu = 0,\quad \forall\;\varphi\in C^{\infty}_{0}(\mathds{R}^m)
\end{equation}
Measures $\mu$ satisfying \eqref{equation mu_diff op} are called \textit{infinitesimally} invariant, or simply invariant if there is no confusion. Let us define
\begin{equation}
\label{measures in kernel}
    \mathcal{M}_{\text{ell}}^{a,b}:=\big\{\mu\;|\; \mu \text{ a probability measure on }\; \mathds{R}^{m}\;\text{satisfying }\; \eqref{equation mu_diff op}\big\}
\end{equation}
where the subscript ``ell" stands for \textit{elliptic}. In \cite{bogachev1999uniqueness}, it is shown that the question whether or not $\mathcal{M}_{\text{ell}}^{a,b}$ contains at most one element turns out to be related to the question whether $\mu\in\mathcal{M}_{\text{ell}}^{a,b}$ is invariant for the $C_{0}$-semigroup generated by the closure of the operator $(L_{a,b},C^{\infty}_{0}(\mathds{R}^m))$. 
In particular, under assumptions that we will shortly make precise, if $\mathcal{M}_{\text{ell}}^{a,b} = \{\mu\}$ a singleton, then $\mu$ allows to define a new operator $(\overline{L}^{\mu}_{a,b}, D(\overline{L}^{\mu}_{a,b}))$ which is the closed extension of $(L_{a,b},C^{\infty}_{0}(\mathds{R}^m))$ on $L^{1}(\mathds{R}^{m};\mu)$.
The latter operator will play a key role in our main result. 

We recall some notations: when a measure $\mu$ has a density $\rho$ with respect to (w.r.t.) Lebesgue measure that we denote by $dx$, then $\mu$ is absolutely continuous w.r.t. $dx$, we write $\mu \ll dx$ and $\rho=\frac{d\mu}{dx}$ is the Radon-Nikodym derivative of $\mu$ w.r.t. $dx$. Let $W^{p,k}(\mathds{R}^{m})$, $p\geq 1, k\geq 0$ be the standard Sobolev space of functions whose generalized derivatives up to order $k$ are in $L^{p}(\mathds{R}^{m})$. If we consider a measure $\mu$ instead of Lebesgue, we write $W^{p,k}(\mathds{R}^{m};\mu)$ to denote the weighted Sobolev space. And let $W^{p,k}_{loc}(\mathds{R}^{m})$ be the class of functions such that $\chi f\in W^{p,k}(\mathds{R}^{m})$ for each $\chi \in C^{\infty}_{0}(\mathds{R}^{m})$ the class of infinitely differentiable functions with compact support in $\mathds{R}^{m}$.

\begin{theorem}(Regularity I --\cite[Theorem 2.1]{bogachev2002uniqueness})
\label{thm regularity meas}
Let $\mu$ be a locally finite and non-negative Borel measure satisfying \eqref{equation mu_diff op}. Assume (A1), (A2) and (A3). Then $\mu \ll dx$ with $\frac{d\mu}{dx}\in W^{p,1}_{\text{loc}}(\mathds{R}^{m})\big(\subset C^{1-\frac{m}{p}}(\mathds{R}^{m})\big)$. If $\rho$ denotes the continuous version of $\frac{d\mu}{dx}$, then for all compact $K\subset \mathds{R}^{m}$, $\exists\;c_{K}\in]0,\infty[$ s.t.: $\sup\limits_{K}\rho \leq c_{K}\inf\limits_{K}\rho$. In particular, either $\rho\equiv 0$ or $\rho(x)>0,\;\forall\;x\in\mathds{R}^{m}$.
\end{theorem}

\begin{proof}
It relies on the results \cite[Corollary 2.10 \& Corollary 2.11]{bogachev2001regularity} which are slightly more general.
\end{proof}

\begin{theorem}(Regularity II --\cite[Theorem 1.1]{bogachev1996regularity})
\label{thm regularity meas 2}
Assume (A1), (A2), (A3), (A4) and (A6) and let $\mu$ be a Borel probability measure solving \eqref{equation mu_diff op}. Then $\mu=\rho \text{d}x$ where $\rho=\varphi^{2}$ and $\varphi\in W^{2,1}(\mathds{R}^{m})$. In particular, we have $|\nabla \rho|^{2}\rho^{-1}\in L^{1}(\mathds{R}^{m})$.
\end{theorem}

\begin{proof}
See Theorem 1.1 and statement (ii) of Remark 2.3 in \cite{bogachev1996regularity}. Note that the theorem in \cite{bogachev1996regularity} requires the drift $|b|$ to be in $L^{2}(\mathds{R}^{m};\mu)$. But as we shall see in Lemma \ref{conv-prob law}, $\mu$ has finite moments of any order and hence $b$ with a polynomial growth (as in (A6)) satisfies this assumption.
\end{proof}

\begin{theorem}(Existence --\cite[Theorem 5.7]{bogachev2002uniqueness})
\label{thm existence inv meas}
Assume (A1), (A2) and (A3). And assume in addition that there exists a function $\omega\in C^{2}(\mathds{R}^{m})$ s.t.
\begin{equation}\label{eq: lypaunov - survey}
    \omega(x) \to +\infty\;\text{ and }\; L_{a,b}\omega(x) \to -\infty\quad \text{ as } |x|\to \infty.
\end{equation}
Then $\mathcal{M}_{\text{ell}}^{a,b}$ as defined in \eqref{measures in kernel} is non-empty. 
\end{theorem}

\begin{proof}
It relies on \cite[Theorem 1.2]{bogachev2000generalization}. See also \cite{bogachev2002uniqueness} and references therein.
\end{proof}

\begin{corollary}\label{cor: survey lyapunov}
Assume (A1), (A2), (A3) and (A4). Then $\omega(x):=|x|^{2}$ fulfills \eqref{eq: lypaunov - survey} and the conclusion of Theorem \ref{thm existence inv meas} holds.
\end{corollary}

\begin{proof}
See the proof of \cite[Corollary 1.4]{bogachev2000generalization} and also \cite[Corollary 1.3(ii)]{bogachev2000generalization}.
\end{proof}

Let us consider now the situation of Theorem \ref{thm existence inv meas}. Fix $\mu\in\mathcal{M}_{\text{ell}}^{a,b}$. As observed in \cite[\S 2.3]{bogachev1999uniqueness}, by Theorem \ref{thm regularity meas}, $\mu$ is equivalent to Lebesgue measure, and therefore is strictly positive on all non-empty open subsets of $\mathds{R}^{m}$. So $C^{\infty}_{0}(\mathds{R}^{m})$ can be identified with a subset of $L^{1}(\mathds{R}^{m};\mu)$, since each corresponding $\mu$-class has a unique continuous $\mu$-version. Hence the operator $(L_{a,b},C^{\infty}_{0}(\mathds{R}^{m}))$ is well defined on $L^{1}(\mathds{R}^{m};\mu)$. The following theorem relies on \textit{dissipativity} and \textit{essential m-dissipativity} of the operator $L_{a,b}$ (see definition in \cite{bogachev2002uniqueness}, end of \S 1). 

\begin{theorem}
\label{thm closed extension}
Assume (A1), (A2), (A3) and (A4). Then $\mathcal{M}^{a,b}_{\text{ell}}=\{\mu\}$ is a singleton and the following statements hold true
\begin{enumerate}[label=(\roman*)]
    \item there exists a closed extension of the operator $(L_{a,b},C^{\infty}_{0}(\mathds{R}^{m}))$ on $L^{1}(\mathds{R}^{m};\mu)$;
    \item its closure $(\overline{L}^{\mu}_{a,b}, D(\overline{L}^{\mu}_{a,b}))$ on $L^{1}(\mathds{R}^{m};\mu)$ generates a  $C_{0}$-semigroup $(T^{\mu}_{t})_{t\geq 0}$ on $L^{1}(\mathds{R}^{m};\mu)$;
    \item $(T^{\mu}_{t})_{t\geq 0}$ is the only $C_{0}$-semigroup on $L^{1}(\mathds{R}^{m};\mu)$ which has a generator extending $(L_{a,b},C^{\infty}_{0}(\mathds{R}^{m}))$;
    \item $(T^{\mu}_{t})_{t\geq 0}$ is contractive, and $\mu$ is $(T^{\mu}_{t})_{t\geq 0}$-invariant in the sense
    \begin{equation}\label{eq: semigroup invariance}
        \int_{\mathds{R}^{m}}T^{\mu}_{t}fd\mu = \int_{\mathds{R}^{m}}fd\mu,\quad \forall\;f\in L^{\infty}(\mathds{R}^{m};\mu).
    \end{equation}
\end{enumerate}
\end{theorem}

\begin{proof}
The set $\mathcal{M}^{a,b}_{\text{ell}}$ being a singleton (i.e. existence and uniqueness of the invariant measure $\mu$) is a consequence of Corollary 1 above and \cite[Example 5.1]{bogachev2002uniqueness}. The other statements rely on the results of Theorem 2.3 and Proposition 2.6 in \cite{bogachev2002uniqueness}.
\end{proof}

Thanks to this result, we can now define on a larger space the operator $\mathcal{L}$ in the problem \eqref{eq: intro - pde}. This is an important step when dealing with unbounded right-hand side terms $f$ in \eqref{eq: intro - pde}, since there cannot exist any solution in $C^{\infty}_{0}(\mathds{R}^{m})$.\\
Indeed, the differential operator $(\mathcal{L}, D(\mathcal{L}))$ in \eqref{eq: intro - pde} should be understood in the sense of the closed extension $(\overline{L}^{\mu}_{a,b},D(\overline{L}^{\mu}_{a,b}))$ provided by Theorem \ref{thm closed extension}, where $D(\overline{L}^{\mu}_{a,b})$ is the closure of $C^{\infty}_{0}(\mathds{R}^m)$ in $L^{1}(\mathds{R}^{m};\mu)$. More precisely, we have $C^{\infty}_{0}(\mathds{R}^m)\subset D(\overline{L}^{\mu}_{a,b})\subset L^{1}(\mathds{R}^{m};\mu)$ with dense inclusions.

In the following, we state from \cite{bogachev2002uniqueness} a theorem which makes $D(\overline{L}^{\mu}_{a,b})$ more precise. In fact, for every $r\in [1,+\infty)$, the restriction of the semigroup $\{T^{\mu}_{t}\}_{t\geq 0}$, whose generator is $\overline{L}^{\mu}_{a,b}$, to $L^{r}(\mathds{R}^{m};\mu)$ is a strongly continuous semigroup on $L^{r}(\mathds{R}^{m};\mu)$ (see \cite[Lemma 5.1.4, p. 180]{bogachev2015fokker}). Its generator will be denoted by $(L^{\mu,r}_{a,b},D(L^{\mu,r}_{a,b}))$, where
\begin{equation*}
    D(L^{\mu,r}_{a,b}) = \{f\in D(L^{\mu}_{a,b})\cap L^{r}(\mathds{R}^{m};\mu): \; L^{\mu}_{a,b}f\in L^{r}(\mathds{R}^{m};\mu)\}.
\end{equation*}

\begin{theorem}(\cite[Theorem 2.8(i)]{bogachev2002uniqueness})\label{thm sobolev domain extension}
In the situation of Theorem \ref{thm closed extension}, one has for any $r\in [1,+\infty)$
\begin{equation}\label{eq: domain extension}
\begin{aligned}
    & (L^{\mu,r}_{a,b},D(L^{\mu,r}_{a,b})) \subset \{f\in L^{r}(\mathds{R}^{m};\mu)\cap W^{r,2}_{\text{loc}}(\mathds{R}^m)\,:\, L_{a,b}f\in L^{r}(\mathds{R}^{m};\mu)\}\\
    & \quad \text{and }\; L^{\mu,r}_{a,b}f=L_{a,b}f\quad \text{for all}\; f\in D(L^{\mu,r}_{a,b}).
\end{aligned}
\end{equation}
\end{theorem}


The next result is from \cite{veretennikov1997polynomial} and concerns the moments of the invariant measure. 
\begin{lemma}\label{conv-prob law}
Assuming (A1), (A3) and (A4), the invariant probability measure $\mu$ exists and has finite moments of any order $\ell\geq1$, i.e. $\int_{\mathds{R}^{m}}|x|^{\ell}\,\text{d}\mu(x) <+\infty$. 
\end{lemma}
\begin{proof}
This is a particular case of the more general result in \cite[Theorem 6]{veretennikov1997polynomial} (see in particular \cite[eq. (28) in \S 6]{veretennikov1997polynomial}). Indeed, the main assumption in \cite{veretennikov1997polynomial} is
\begin{equation}
\label{assumption-ver}
    \exists\;M_{0}\geq 0,\; r\geq 0 \;\text{ s.t. }\quad \langle b(x),x \rangle \leq -r,\quad \forall\;|x|\geq M_{0}
\end{equation}
Then introduce the following constants
\begin{equation*}
    \begin{aligned}
        & \lambda_{-}:=\inf\limits_{x\neq 0}\; \langle \varrho\varrho^{*}(x)\frac{x}{|x|},\frac{x}{|x|} \rangle,\quad
        & \lambda_{+} := \sup\limits_{x\neq 0}\;\langle \varrho\varrho^{\top}(x)\frac{x}{|x|},\frac{x}{|x|} \rangle\\
        & \Tilde{\Lambda} := \sup\limits_{x}\frac{\text{trace}(\varrho\varrho^{\top}(x))}{m}, \quad
        & r_{0} := [r - (m\Tilde{\Lambda} - \lambda_{-})/2]\lambda_{+}^{-1}
    \end{aligned}
\end{equation*}
In this context, it is shown that the invariant measure has finite moments of order $\ell\in(2k+2, 2r_{0}-1)$, where again $k\in (0,r_{0}-\frac{3}{2})$. 
In our case, assumption (A4) guarantees a constant $r$ (in \eqref{assumption-ver}) as large as we want. Then using H\"older inequality together with the fact that $\mu(\mathds{R}^{m})=1$, one proves finite moments of any order $\ell\geq 1$.
\end{proof}

\begin{remark}\label{rmk: weak assumptions}
The growth condition (A5) can be replaced by an integrability condition with respect to the invariant measure, i.e. $f\in L^{1}(\mathds{R}^{m};\mu)$. Thus, also some exponentially growing functions $f$ can be handled analogously.  
Indeed, if $b(x) = -\tilde{\gamma}\,x$ for some $\Tilde{\gamma}>0$ and $a(x)=I$ the identity matrix, then the stochastic process is an Ornstein-Uhlenbeck whose invariant (Gibbs) measure behaves as $e^{-\tilde{\gamma}|x|^{2}}$ and allows to perform the subsequent computations.
\end{remark}

\subsection{Distance between invariant measures}
\label{sec:distance inv}
The results of this subsection are taken from \cite{bogachev2018poisson} and concern the distance between two stationary solutions (invariant probability measures) to FPK equation \eqref{equation mu_diff op} with different diffusion coefficients $a_{\alpha} = (a^{ij}_{\alpha}), a_{\beta} = (a^{ij}_{\beta})$ and drift coefficients $b_{\alpha} = (b^{i}_{\alpha}), b_{\beta} = (b^{i}_{\beta})$. 
We denote by $\mu_{\alpha},\mu_{\beta}$ the two Borel probability measures solving $L^{*}_{a_{\alpha},b_{\alpha}}\mu_{\alpha}=0$ and $L^{*}_{a_{\beta},b_{\beta}}\mu_{\beta}=0$ as discussed in \S \ref{sec:extension diff}, and by $\rho_{\alpha},\rho_{\beta}$ their corresponding continuous densities. We introduce the following notation as in \cite{bogachev2018poisson}:
\begin{equation*}
    h_{\alpha} = (h^{i}_{\alpha})^{m}_{i=1},\quad h^{i}_{\alpha} = b^{i}_{\alpha} - \sum\limits_{j=1}^{m}\partial_{x_{j}}a^{ij}_{\alpha}, \text{ and } \quad h_{\beta} = (h^{i}_{\beta})^{m}_{i=1},\quad h^{i}_{\beta} = b^{i}_{\beta} - \sum\limits_{j=1}^{m}\partial_{x_{j}}a^{ij}_{\beta},
\end{equation*}
and define \vspace*{-3mm}
\begin{equation*}
    \Phi = (a_{\beta} - a_{\alpha})\frac{\nabla \rho_{\alpha}}{\rho_{\alpha}} + h_{\alpha} - h_{\beta}.
\end{equation*}
Note that $\Phi = b_{\alpha}  -b_{\beta}$ if $a_{\alpha}=a_{\beta}$. We also denote by $\|\cdot\|_{TV}$ the total variation norm of a signed measure, defined as the sum of values on the whole space of its positive and negative parts. Our $\mu_{\alpha},\mu_{\beta}$ are denoted in \cite{bogachev2018poisson} by $\sigma,\mu$ respectively.

\begin{theorem}(\cite[Theorem 3.2]{bogachev2018poisson})\label{thm distance}
Assume (A1), (A2), (A3), (A4) and (A6). If $(1+|x|)^{\theta}|\Phi|\in L^{1}(\mathds{R}^{m};\mu_{\alpha})$, where $\theta$ is as in (A6), then 
\begin{equation*}
    \|\mu_{\alpha} - \mu_{\beta}\|_{TV} \leq C \int_{\mathds{R}^{m}}(1+|x|)^{\theta}|\Phi|\text{d}\mu_{\alpha}
\end{equation*}
where $C$ depends on the constants in the assumptions and on the diffusion matrix $a_{\beta}$. 
\end{theorem}

\section{Preliminary results}\label{sec: prelim}

We start by an \textit{exchange property} that we will need in the sequel. We denote by $\mathcal{M}_{d}(\mathds{R}^{m})$ the space of totally finite signed Borel measures on $\mathds{R}^{m}$ with finite $d$-moment, i.e. for any $\mu\in \mathcal{M}_{d}(\mathds{R}^{m})$, one has $\int_{\mathds{R}^{m}}|x|^{d}\,\text{d}|\mu|(x)<+\infty$ where $|\mu|=\mu^{+}+\mu^{-} = \|\mu\|_{TV}$ and $\mu^{+},\mu^{-}$ are the positive and negative parts of $\mu$, and by $\mathcal{M}_{d}^{+}(\mathds{R}^{m})$ the subspace of non-negative measures.
\begin{proposition}\label{prop: exchange prop}
Let $f$ satisfies (A5). The following holds for any $q\in \mathcal{M}_{d}^{+}(\mathds{R}^{m})$\vspace*{-1mm}
\begin{equation}
    \label{eq: exchange prop}
    \int_{\mathds{R}^{m}}\min\limits_{\alpha\in A} f(x,\alpha)\,\text{d}q(x) = \min\limits_{\upalpha(\cdot)\in\mathcal{A}}\int_{\mathds{R}^{m}}f(x,\upalpha(x))\,\text{d}q(x)\vspace*{-1mm}
\end{equation}
where $A$ is a compact subset of $\mathds{R}^{k}$, for some $k>0$, and $\mathcal{A}$ is the set of measurable functions $\upalpha(\cdot):\mathds{R}^{m}\to A$. And the same holds true with $\max$ instead of $\min$.
\end{proposition}

\begin{remark}  In the context of stochastic control, the set $\mathcal{A}$ needs to be the one of progressively measurable functions. In fact, these are the admissible controls.
\end{remark}

\begin{proof}
Let $q\in \mathcal{M}_{d}^{+}(\mathds{R}^{m})$ be arbitrarily fixed and $f:\mathds{R}^{m}\times A\to \mathds{R}$ satisfies (A5).\\
To prove the inequality ``$\,\leq\,$", it suffices to observe that for any $\varepsilon>0$, there exists $\upalpha^{\varepsilon}(\cdot)\in \mathcal{A}$ such that
\begin{equation*}
    \begin{aligned}
    \min\limits_{\upalpha(\cdot)\in\mathcal{A}}\int_{\mathds{R}^{m}}f(x,\upalpha(x))\,\text{d}q(x) + \varepsilon\; & \geq \int_{\mathds{R}^{m}}f(x,\upalpha^{\varepsilon}(x))\,\text{d}q(x) \geq \int_{\mathds{R}^{m}}\,\min\limits_{\alpha\in A}f(x,\alpha)\,\text{d}q(x)
    \end{aligned}
\end{equation*}
and hence the result. 
To prove the inequality ``$\geq$", we consider the minimization problem $\mathfrak{f}(x) := \min\limits_{\alpha\in A}f(x,\alpha)$ where $x\in \mathds{R}^{m}$. Since $A$ is compact and $\mathfrak{f}(x)\in f(\{x\}\times A)$ with $\mathfrak{f}$ measurable and $f(x,\alpha)$ is measurable in $x$ and continuous in $\alpha$, then 
a classical selection theorem (see \cite[Theorem 7.1, p. 66]{Himmelberg1975}) implies the existence of a measurable selector $\overline{\upalpha}$ for which the minimization is achieved, i.e.
\begin{equation*}
    \exists\;\overline{\upalpha}\in \mathcal{A},\;\text{s.t. } \;\forall \; x \in \mathds{R}^{m},\; \mathfrak{f}(x) = \min\limits_{\alpha\in A}f(x,\alpha) = f(x,\overline{\upalpha}(x)).
\end{equation*}
Therefore one has
\begin{equation*}
    \begin{aligned}
    \int_{\mathds{R}^{m}}\min\limits_{\alpha\in A}f(x,\alpha)\,\text{d}q(x) & = \int_{\mathds{R}^{m}}f(x,\overline{\upalpha}(x))\,\text{d}q(x)\; \geq \min\limits_{\upalpha(\cdot)\in\mathcal{A}} \int_{\mathds{R}^{m}}f(x,\upalpha(\cdot))\,\text{d}q(x).
    \end{aligned}
\end{equation*}
This yields the second desired inequality and concludes the proof. 
\end{proof}

The \textit{exchange property} in Proposition \ref{prop: exchange prop} ensures that we can exchange the minimization over the parameters $\alpha$ and the duality product in $\mathcal{M}_{d}(\mathds{R}^{m})$ provided we define the second argument in $f$ as measurable functions $\upalpha(\cdot)\in\mathcal{A} =L^{\infty}(\mathds{R}^{m},A)$ instead of vectors $\alpha\in A$, that is,
\begin{equation*}
    \min\limits_{\upalpha(\cdot)\in \mathcal{A}}\;\langle\; f(\cdot\,, \upalpha(\cdot))\;,\; q\;\rangle = \langle\; \min\limits_{\alpha\in A} f(\cdot\,,\alpha)\;,\; q\; \rangle
\end{equation*}

The next result concerns the continuity of the functional $F:\mathcal{A}\to \mathds{R}$ defined by 
\begin{equation}
    \label{functional F}
    F(\upalpha) = \int_{\mathds{R}^{m}}f(x,\upalpha(x))\text{d}\mu_{\upalpha}(x) = \langle f(\cdot\,\upalpha(\cdot),\mu_{\upalpha} \rangle
\end{equation}
where $\mathcal{A}$ is endowed with its weak-$*$ topology, and $\mu_{\upalpha}$ is the unique invariant probability measure satisfying the FPK equation $\mathcal{L}^{*}_{\upalpha}\mu_{\upalpha}=0$ where $\mathcal{L}^{*}_{\upalpha}$ is the formal adjoint operator to second order elliptic operator
\begin{equation*}
    \mathcal{L}_{\upalpha}\varphi(x) = \text{trace} \big( a(x,\upalpha(x))D\varphi(x) \big) + b(x,\upalpha(x))\cdot \nabla \varphi(x).
\end{equation*}
We recall that existence, uniqueness and regularity of $\mu_{\upalpha}$ have been discussed in \S \ref{sec:extension diff}. \\
For the sake of precision, we state the following definitions. 
\begin{definition} 
Noting that $\mathcal{A}=L^{\infty}(\mathds{R}^{m},A)\subset L^{\infty}(\mathds{R}^{m})=(L^{1}(\mathds{R}^{m}))^{*}$, we say\\
    \textbullet \, the map $\upalpha(\cdot)\mapsto g(\cdot\,,\upalpha(\cdot)) \in  L^{\ell}(\mathds{R}^{m};\mu_{\upalpha})$, $\ell\geq 1$, is weak-$*$ continuous at $\upalpha$ if 
    \begin{equation*}
    \begin{aligned}
        &\forall\,\varepsilon>0,\; \exists\, \delta>0 \text{ and a finite collection } \{\xi_{1}, \dots,\xi_{n}\} \text{ from } L^{1}(\mathds{R}^{m}) \text{ such that }\\
        & \quad \; \forall\, \upbeta(\cdot)\in \mathcal{A} \text{ satisfying } 
        \left|\int_{\mathds{R}^{m}}(\upalpha(x) - \upbeta(x))\xi_{i}(x)\text{d}x\right|<\delta \text{ for } i=1, \dots,n, \\
        & \quad \quad \quad \quad \quad \text{ we have } \|g(\cdot\,,\upalpha(\cdot)) - g(\cdot\,,\upbeta(\cdot))\|_{L^{\ell}(\mathds{R}^{m};\mu_{\upalpha})}<\varepsilon
    \end{aligned}
    \end{equation*}
    \textbullet \, the functional $\upalpha(\cdot)\mapsto \mathcal{G}(\upalpha)\in \mathds{R}$ is weak-$*$ continuous at $\upalpha$ if in the last line of the above definition we have $|\mathcal{G}(\upalpha) - \mathcal{G}(\upbeta)|<\varepsilon$.
\end{definition}

We now introduce additional assumptions, for all $1\leq i,j\leq m$
\begin{description}
    \item[B1]{The map 
    $\upalpha(\cdot)\mapsto f(\cdot\, , \upalpha(\cdot))$ is weak-$*$ continuous from $\mathcal{A}$ to $L^{1}(\mathds{R}^{m};\mu_{\upalpha})$,
    }
    \item[B2]{
    The maps 
    $\upalpha(\cdot)\mapsto a^{ij}(\cdot\, , \upalpha(\cdot))$ is weak-$*$ continuous from $\mathcal{A}$ to $L^{4}(\mathds{R}^{m};\mu_{\upalpha})$,  
    }
    \item[B3]{
    The maps 
    $\upalpha(\cdot)\mapsto \partial_{x_{j}} a^{ij}(\cdot\, , \upalpha(\cdot))$ has a polynomial growth and is weak-$*$ continuous from $\mathcal{A}$ to $L^{2}(\mathds{R}^{m};\mu_{\upalpha})$. 
    The notation $\partial_{x_{j}} a^{ij}(\cdot\, , \upalpha(\cdot))$ means the derivative w.r.t. the $j$-th component of the first argument,
    }
    \item[B4]{
    The maps 
    $\upalpha(\cdot)\mapsto b^{i}(\cdot\, , \upalpha(\cdot))$ is weak-$*$ continuous from $\mathcal{A}$ to $L^{2}(\mathds{R}^{m};\mu_{\upalpha})$.
    }
\end{description}

\hfill\\
These assumptions are satisfied for example when $\phi=f,a^{ij},\partial_{x_{j}}a^{ij},b^{i}$ is such that
\begin{equation*}
    |\phi(x,\alpha) - \phi(x,\beta)| \leq k(x) |\alpha - \beta|^{r},\quad \forall \, x\in\mathds{R}^{m},\, \alpha,\beta\in A
\end{equation*}
where $r> 0$ and $k(\cdot)$ has any polynomial growth. 
Indeed, let $\ell \geq 1$, we have
\begin{equation*}
    \begin{aligned}
        \|\phi(\cdot\,,\upalpha(\cdot)) - \phi(\cdot\,,\upbeta(\cdot))\|^{\ell}_{L^{\ell}(\mathds{R}^{m};\mu_{\upalpha})} 
        & \leq \int_{\mathds{R}^{m}} |k(x)|^{\ell}|\upalpha(x) - \upbeta(x)|^{r\ell}\,\text{d}\mu_{\upalpha}(x)\\
        & \leq \|k\|^{\ell}_{L^{\ell q}(\mathds{R}^{m};\mu_{\upalpha})} \left(\int_{\mathds{R}^{m}}|\alpha(x) - \beta(x)|^{\ell r p}\,\text{d}\mu_{\upalpha}(x)\right)^{\frac{1}{p}}
    \end{aligned} 
\end{equation*}
for all $p,q> 1$ such that $1/p + 1/q=1$, and  $\|k\|_{L^{\ell q}(\mathds{R}^{m};\mu_{\upalpha})}$ is finite since  $\mu_{\upalpha}$ has all its moments finite (Lemma \ref{conv-prob law}). Choosing $p$ large such that $\ell r p\geq 2$ yields
\begin{equation*}
    \begin{aligned}
        \|\phi(\cdot\,,\upalpha(\cdot)) - \phi(\cdot\,,\upbeta(\cdot))\|^{p\ell}_{L^{\ell}(\mathds{R}^{m};\mu_{\upalpha})} 
        & \leq \|k\|^{p\ell}_{L^{\ell q}(\mathds{R}^{m};\mu_{\upalpha})} \int_{\mathds{R}^{m}}|\upalpha(x) - \upbeta(x)|\psi(x)\,\text{d}x
    \end{aligned} 
\end{equation*}
where $\psi(x) := |\upalpha(x) - \upbeta(x)|^{\ell r p-1}\rho_{\upalpha}(x)$ and $\rho_{\upalpha}$ is the density of $\mu_{\upalpha}$. Noting $\ell r p -1\geq 1$, the triangle inequality yields $\sup_{x}|\upalpha(x) - \upbeta(x)|^{\ell r p - 1}\leq 2\,\text{diam}(A)^{\ell r p - 1}=:C$ and $|\psi(x)| \leq C \rho_{\upalpha}(x)$. Hence, we have
\begin{equation*}
    \begin{aligned}
        \|\phi(\cdot\,,\upalpha(\cdot)) - \phi(\cdot\,,\upbeta(\cdot))\|^{p\ell}_{L^{\ell}(\mathds{R}^{m};\mu_{\upalpha})} 
        & \leq C \|k\|^{p\ell}_{L^{\ell q}(\mathds{R}^{m};\mu_{\upalpha})} \int_{\mathds{R}^{m}}|\upalpha(x) - \upbeta(x)|\rho_{\upalpha}(x)\,\text{d}x.
    \end{aligned} 
\end{equation*}
Hence, for any sequence $(\upalpha_{n})_{n}$ weak-$*$ converging to $\upalpha$, we have 
\begin{equation*}
    \begin{aligned}
        \|\phi(\cdot\,,\upalpha(\cdot)) - \phi(\cdot\,,\upalpha_{n}(\cdot))\|^{p\ell}_{L^{\ell}(\mathds{R}^{m};\mu_{\upalpha})}
        & \leq C' \int_{\mathds{R}^{m}} |\upalpha(x)  -\upalpha_{n}(x)|\rho_{\upalpha}(x)\,\text{d}x \xrightarrow[n\to\infty]{} 0
    \end{aligned}
\end{equation*}
and the map $\upalpha(\cdot)\mapsto\phi(\cdot\,,\upalpha(\cdot))$ is weak-$*$ continuous from $\mathcal{A}$ to $L^{\ell}(\mathds{R}^{m};\mu_{\upalpha})$. \\
In the notation of the above Definition, the finite collection from $L^{1}(\mathds{R}^{m})$ is a singleton made of $\rho_{\upalpha}(\cdot)\in L^{1}(\mathds{R}^{m})$ as it is the density of a probability measure. 

We need the matrix-norm: for a matrix function $M=(M^{ij})\in \mathds{R}^{p\times q}$, $p,q\geq 1$, we write $|M(x)|:=\max_{1\leq i\leq p}\sum_{j=1}^{q}|M^{ij}(x)|$ and $\| \,|M|\,\|^{\ell}_{L^{\ell}(\mathds{R}^{m};\mu)} = \int |M(x)|^{\ell}\text{d}\mu$.

For simplicity of notation, we write the functions $a_{\upalpha} = a(\cdot,\upalpha(\cdot))$, $b_{\alpha}=b(\cdot,\upalpha(\cdot))$,  $f(\upalpha)=f(\cdot,\upalpha(\cdot))$ and the weighted Lebesgue space $L^{\ell}_{\mu_{\upalpha}} := L^{\ell}(\mathds{R}^{m};\mu_{\upalpha})$.

\begin{proposition}\label{prop: continuity F}
Assume (A1-A6) and (B1-B4) are satisfied. Then the functional $F$ defined in \eqref{functional F} is weak-$*$ continuous. Moreover, we have $\forall \,\upalpha(\cdot),\upbeta(\cdot)\in \mathcal{A}$
\begin{equation*}
\begin{aligned}
    |F(\upalpha) - F(\upbeta)| \leq & \; C\left( \| \,|b_{\upalpha} - b_{\upbeta} |\,\|_{_{L^{2}_{\mu_{\upalpha}}}} 
    + \| \,|\nabla a_{\upalpha} - \nabla a_{\beta}|\, \|_{_{L^{2}_{\mu_{\upalpha}}}}
    + \| \,|a_{\upalpha} - a_{\upbeta}|\,\|_{_{L^{4}_{\mu_{\upalpha}}}}
    \right)^{\frac{1}{2}}\\
    & \quad \quad + \|f(\upalpha) - f(\upbeta)\|_{_{L^{1}_{\mu_{\upalpha}}}}
\end{aligned}
\end{equation*}
where $\nabla a_{\upalpha}$ is a vector whose $i$-th component is $\sum_{j=1}^{m}\partial_{x_{j}}a^{ij}_{\upalpha}(x)$, and 
$C>0$ is a constant depending on the parameters in the assumptions, on the diffusion matrix $a_{\upbeta}$ and on $\mu_{\upalpha},\mu_{\beta}$. 
\end{proposition}

\begin{proof}
Given $\upalpha,\upbeta\in \mathcal{A}$, we have
\begin{equation}\label{first ineq proof}
    \begin{aligned}
        F(\upalpha) - F(\upbeta) & = \langle f(\cdot\,,\upalpha(\cdot)),\mu_{\upalpha} \rangle - \langle f(\cdot\,,\upbeta(\cdot)), \mu_{\upbeta} \rangle \\
        & = \langle f(\cdot\,,\upalpha(\cdot)) - f(\cdot\,,\upbeta(\cdot)),\mu_{\upalpha} \rangle + \langle f(\cdot\,,\upbeta(\cdot)), \mu_{\upalpha} - \mu_{\upbeta} \rangle \\
        & \leq \|f(\cdot\,,\upalpha(\cdot)) - f(\cdot\,,\upbeta(\cdot))\|_{_{L^{1}_{\mu_{\upalpha}}}} + \langle f(\cdot\,,\upbeta(\cdot)), \mu_{\upalpha} - \mu_{\upbeta} \rangle.
    \end{aligned}
\end{equation}
We need to estimate the second term in the r.h.s. of the above. Recall  $\text{d}\mu_{\upalpha}(x) = \rho_{\upalpha}(x)\text{d}x$ and $\text{d}\mu_{\upbeta}(x) = \rho_{\upbeta}(x)\text{d}x$ with $\rho_{\upalpha},\rho_{\upbeta}>0$ (Theorem \ref{thm regularity meas} \& Theorem \ref{thm regularity meas 2}). We have
\begin{equation*}
    \begin{aligned}
        & \langle f(\cdot\,,\upbeta(\cdot)), \mu_{\upalpha} - \mu_{\upbeta} \rangle  = \int_{\mathds{R}^{m}} f(x,\upbeta(x))\big(\rho_{\upalpha}(x) - \rho_{\beta}(x)\big)\text{d}x\\
        & \quad \quad \quad  = \int f(x,\upbeta(x))\left(\sqrt{\rho_{\upalpha}(x)} + \sqrt{\rho_{\upbeta}(x)}\right)\left(\sqrt{\rho_{\upalpha}(x)} - \sqrt{\rho_{\upbeta}(x)}\right)\text{d}x\\
        & \quad \quad \quad \leq \int |f(x,\upbeta(x))|\sqrt{\rho_{\upalpha}(x)}\left|\sqrt{\rho_{\upalpha}(x)} - \sqrt{\rho_{\upbeta}(x)}\right|\,\text{d}x \\
        & \quad \quad \quad \quad \quad \quad \quad \quad \quad + \int |f(x,\upbeta(x))|\sqrt{\rho_{\upbeta}(x)}\left|\sqrt{\rho_{\upalpha}(x)} - \sqrt{\rho_{\upbeta}(x)}\right|\,\text{d}x\\
        & \quad \quad \quad \leq \left(\int |f(x,\upbeta(x))|^{2}\text{d}\mu_{\upalpha}(x)\right)^{\frac{1}{2}}\left(\int \left(\sqrt{\rho_{\upalpha}(x)} - \sqrt{\rho_{\upbeta}(x)}\right)^{2}\text{d}x\right)^{\frac{1}{2}}\\
        & \quad \quad \quad \quad \quad \quad \quad \quad \quad + \left(\int |f(x,\upbeta(x))|^{2}\text{d}\mu_{\upbeta}(x)\right)^{\frac{1}{2}}\left(\int \left(\sqrt{\rho_{\upalpha}(x)} - \sqrt{\rho_{\upbeta}(x)}\right)^{2}\text{d}x\right)^{\frac{1}{2}}\\
        &\quad \quad \quad = \sqrt{2} \left(\|f(\upbeta)\|_{L^{2}_{\mu_{\upalpha}}} + \|f(\upbeta)\|_{L^{2}_{\mu_{\upbeta}}}\right)H\left(\mu_{\upalpha},\mu_{\upbeta}\right)
    \end{aligned}
\end{equation*}
where $H(\cdot\,,\cdot)$ is the Hellinger distance\footnote{It is defined by $H^{2}(\mu_{\upalpha},\mu_{\upbeta}) = \frac{1}{2}\int \left(\sqrt{\rho_{\upalpha}(x)} - \sqrt{\rho_{\upbeta}(x)}\right)^{2}\text{d}x.$
} between two probability densities. It is known\footnote{This is a consequence of $(\sqrt{A}-\sqrt{B})^{2} = A+B-A\wedge B\leq |A-B|$. In fact, $H(\cdot\,,\cdot)$ is topologically equivalent to the total variation distance.} that $H^{2}(\mu_{\upalpha},\mu_{\upbeta}) \leq \|\mu_{\upalpha} - \mu_{\upbeta}\|_{_{TV}}$. Therefore we have
\begin{equation*}
    \begin{aligned}
        |\langle f(\cdot\,,\upbeta(\cdot)), \mu_{\upalpha} - \mu_{\upbeta} \rangle| \leq \sqrt{2} \left(\|f(\upbeta)\|_{L^{2}_{\mu_{\upalpha}}} + \|f(\upbeta)\|_{L^{2}_{\mu_{\upbeta}}}\right)\|\mu_{\upalpha} - \mu_{\upbeta}\|_{_{TV}}^{\frac{1}{2}}.
    \end{aligned}
\end{equation*}
Using Theorem \ref{thm distance} yields
\begin{equation}\label{second ineq proof}
    \begin{aligned}
        & |\langle f(\cdot\,,\upbeta(\cdot)), \mu_{\upalpha} \!-\! \mu_{\upbeta} \rangle|  \leq C \!\left(\|f(\upbeta)\|_{L^{2}_{\mu_{\upalpha}}} +  \|f(\upbeta)\|_{L^{2}_{\mu_{\upbeta}}}\right) \left(\int(1+|x|)^{\theta}|\Phi|\text{d}\mu_{\alpha}\!\right)^{\!\frac{1}{2}}
    \end{aligned}
\end{equation}
if the integral term is well defined (recall the definition of $\Phi$ in \S\ref{sec:distance inv}). Indeed we have
\begin{equation}\label{Phi in proof}
    \begin{aligned}
        & \int_{\mathds{R}^{m}}(1+|x|)^{\theta}|\Phi|\text{d}\mu_{\alpha}  = \int_{\mathds{R}^{m}}(1+|x|)^{\theta}\left| (a_{\upbeta} - a_{\upalpha})\frac{\nabla \rho_{\upalpha}}{\rho_{\upalpha}} + h_{\upalpha} - h_{\upbeta} \right|\rho_{\upalpha}\text{d}x \\
        & \quad \quad \quad \leq \int (1+|x|)^{\theta}\left| (a_{\upbeta} - a_{\upalpha})\frac{\nabla \rho_{\upalpha}}{\rho_{\upalpha}}\right|\rho_{\upalpha}\text{d}x + \int (1+|x|)^{\theta}\left|h_{\upalpha} - h_{\upbeta} \right|\rho_{\upalpha}\text{d}x.
    \end{aligned}
\end{equation}
To estimate the first integral, we write $\frac{|\nabla \rho_{\upalpha}|}{\rho_{\upalpha}}\rho_{\upalpha}\text{d}x=\frac{|\nabla \rho_{\upalpha}|}{\sqrt{\rho_{\upalpha}}}\sqrt{\rho_{\upalpha}}\text{d}x$ then use Cauchy-Schwarz inequality and the last statement in Theorem \ref{thm regularity meas 2}. Recall the matrix-norm $|a(x)|:=\max_{1\leq i\leq m}\sum_{j=1}^{m}|a^{ij}(x)|$. We have
\begin{equation*}
    \begin{aligned}
        & \int (1+|x|)^{\theta}\left| (a_{\upbeta} - a_{\upalpha}\frac{\nabla \rho_{\upalpha}}{\rho_{\upalpha}}\right|\rho_{\upalpha}\text{d}x \leq  \int \frac{|\nabla \rho_{\upalpha}|}{\sqrt{\rho_{\upalpha}}} (1+|x|)^{\theta}\left|a_{\upbeta}(x)-a_{\upalpha}(x)\right|\sqrt{\rho_{\upalpha}}\text{d}x \\ 
        & \quad \quad \quad \leq \left(\int \frac{|\nabla \rho_{\upalpha}|^{2}}{\rho_{\upalpha}}\text{d}x\right)^{\frac{1}{2}}\left(\int (1+|x|)^{2\theta}\left| a_{\upbeta}(x) - a_{\upalpha}(x)\right|^{2}\rho_{\upalpha}\text{d}x\right)^{\frac{1}{2}}\\
        & \quad \quad \quad \leq \left(\int \frac{|\nabla \rho_{\upalpha}|^{2}}{\rho_{\upalpha}}\text{d}x\right)^{\frac{1}{2}} \left(\int (1+|x|)^{4\theta}\rho_{\upalpha}\text{d}x\right)^{\frac{1}{4}}\left(\int \left| a_{\upbeta}(x) - a_{\upalpha}(x)\right|^{4}\rho_{\upalpha}\text{d}x\right)^{\frac{1}{4}}\\
        & \quad \quad \quad = m_{1}(\upalpha)\,\|\,|a_{\upalpha} -a_{\upbeta}|\,\|_{L^{4}_{\mu_{\upalpha}}}
    \end{aligned}
\end{equation*}
where $m_{1}(\upalpha) := \big\||\nabla \rho_{\upalpha}|^{2}\rho_{\upalpha}^{-1}\big\|_{_{L^{1}(\mathds{R}^{m})}}\left(\int (1+|x|)^{4\theta}\rho_{\upalpha}\text{d}x\right)^{\frac{1}{4}}$ is a positive constant, recalling from Lemma \ref{conv-prob law} that $\rho_{\upalpha}$ has all its moments finite.\\
To estimate the second term in \eqref{Phi in proof}, recall $\nabla a_{\upalpha}$ is a vector whose $i$-th component is $\sum_{j=1}^{m}\partial_{x_{j}}a^{ij}_{\upalpha}(x)$, so we have
\begin{equation*}
    \begin{aligned}
        &\int(1+|x|)^{\theta}|h_{\upalpha} - h_{\upbeta}|\rho_{\upalpha}\text{d}x\leq \int(1+|x|)^{\theta}(|b_{\upalpha} - b_{\upbeta}| + |\nabla a_{\upalpha} - \nabla a_{\upbeta}|)\rho_{\upalpha}\text{d}x \\
        & \quad \quad \quad \leq m_{2}(\upalpha)\left( \left(\int |b_{\upalpha}  -b_{\upbeta}|^{2}\rho_{\upalpha}\text{d}x\right)^{\frac{1}{2}} + \left(\int|\nabla a_{\upalpha} - \nabla a_{\upbeta}|^{2}\rho_{\upalpha}\text{d}x\right)^{\frac{1}{2}} \right)\\
        & \quad \quad \quad = m_{2}(\upalpha) \left( \|\, |b_{\upalpha} - b_{\upbeta} | \,\|_{L^{2}_{\mu_{\upalpha}}} + \|\, |\nabla a_{\upalpha} - \nabla a_{\upbeta} | \,\|_{L^{2}_{\mu_{\upalpha}}} \right)
    \end{aligned}
\end{equation*}
where $m_{2}(\upalpha):=\left(\int (1+|x|)^{2\theta}\rho_{\upalpha}\text{d}x\right)^{\frac{1}{2}}$. Therefore we have, for $m_{3} = m_{1}\vee m_{2}$,
\begin{equation*}
\begin{aligned}
    & \int_{\mathds{R}^{m}}(1+|x|)^{\theta}|\Phi|\text{d}\mu_{\alpha}\\  
    & \quad \quad \quad \quad \leq m_{3}(\upalpha) \left( \|\, |b_{\upalpha} - b_{\upbeta} | \,\|_{L^{2}_{\mu_{\upalpha}}} + \|\, |\nabla a_{\upalpha} - \nabla a_{\upbeta} | \,\|_{L^{2}_{\mu_{\upalpha}}} + \|\,|a_{\upalpha} -a_{\upbeta}|\,\|_{L^{4}_{\mu_{\upalpha}}} \right).
\end{aligned}
\end{equation*}
All the terms in the r.h.s. of the above inequality are well defined ($b_{\upalpha},b_{\upbeta}$ and $\nabla a_{\upalpha},\nabla a_{\upbeta}$ have polynomial growth, $a_{\upalpha},a_{\upbeta}$ are bounded, and $\mu_{\upalpha}$ has all its moments finite), so the integral in the l.h.s. is finite. \\
Finally, 
using the latter inequality in \eqref{second ineq proof}, and then in \eqref{first ineq proof}, yields the desired estimate and concludes the proof of weak-$*$ continuity. 
\end{proof}

\section{Main results}\label{sec: main NL}

\subsection{The primal problem}\label{sec: primal NL}

We are interested in a class of fully nonlinear equations, usually called \textit{ergodic (stationary) Hamilton-Jacobi-Bellman} (HJB) equations, and the corresponding ergodic problem is the following
\begin{equation} \label{cell prob - NL - 1}
    \textit{Find $(c,u(\cdot))\in\mathds{R}\times \mathcal{X}(\mathds{R}^{m})$ s.t.: } \; H(x,\nabla u(x),D^{2}u(x)) = c,\;\text{ in }\mathds{R}^{m}
\end{equation}
where $\mathcal{X}$ is a functional space (part of the unknowns), the Bellman Hamiltonian is\vspace*{-2mm}
\begin{equation*}
        H \coloneqq \min\limits_{\alpha\in A}\{\,-\mathcal{L}_{\alpha}u(x) + f(x,\alpha)\,\}\vspace*{-2mm}
    \end{equation*}
and for each $\alpha\in A$ compact subset of $\mathds{R}^{k}$ with $k>0$, the linear differential operator $\mathcal{L}_{\alpha}\varphi(x) := \mathcal{L}_{\alpha}(x,\nabla \varphi(x),D^{2}\varphi(x))$ is defined by
\begin{equation}\label{eq: diff op - NL}
    \mathcal{L}_{\alpha}\varphi(x) =  \text{trace}(a(x,\alpha)D^{2}\varphi(x)) + b(x,\alpha)\cdot\nabla\varphi(x),
    \quad \varphi \in C^{\infty}_{0}(\mathds{R}^m).
\end{equation}

We state our \textit{primal} problem as follows
\begin{equation}
    \label{eq: primal - NL}
    \tag{$\mathfrak{P}$}
    \min\limits_{q\in\mathcal{M}_{d}^{+}(\mathds{R}^{m})}\left\{\,\min\limits_{\upalpha(\cdot)\in \mathcal{A}}\quad  \langle f(\cdot\,,\upalpha(\cdot)),q\rangle,\quad \text{s.t.: }\; 1-\langle 1,q\rangle = 0 \;\text{and }\, q\in \text{Ker}(\mathcal{L}_{\upalpha}^{*})\right\}
\end{equation}
where we recall $\langle f(\cdot\,,\upalpha(\cdot)),q\rangle = \int_{\mathds{R}^m}f(x,\upalpha(x))\text{d}q(x)$, and $q\in \text{Ker}(\mathcal{L}_{\upalpha}^{*})$ is understood in the sense \eqref{equation mu_diff op}-\eqref{sense equation mu_diff op}. We will use the same notation as in \S\ref{sec:duality theory} that we recall here for the reader's convenience 
\begin{gather*}
    X=\mathcal{M}_{d}(\mathds{R}^{m})\quad \text{and} \quad Q=\mathcal{M}_{d}^{+}(\mathds{R}^{m})\\
    G_{1}:X \to \mathds{R},\quad \text{s.t.}\quad G_{1}(q)=1-\langle 1,q\rangle\\
    G_{2}:X\to X,\quad \text{s.t.}\quad G_{2}(q)=q\\
    G=(G_{1},G_{2})\quad \text{and} \quad Y=\mathds{R}\times X\\
    K_{1}=\{0\},\; K_{2}(\upalpha)=\text{Ker}(\mathcal{L}^{*}_{\upalpha})\quad \text{and} \quad K_{\upalpha}=K_{1}\times K_{2}(\upalpha)\subset Y
\end{gather*}
The \textit{primal} problem can then be expressed as
\begin{equation}
    \label{eq: primal - NL  2}
    \tag{$\mathfrak{P}$}
    \min\limits_{q\in Q}\left\{\,\min\limits_{\upalpha(\cdot)\in \mathcal{A}}\quad  \langle f(\cdot\,,\upalpha(\cdot)),q\rangle,\quad \text{s.t.: }\; G(q)\in K_{\upalpha}\right\}.
\end{equation}
The next result shows that the \textit{primal} problem has a solution. 

\begin{lemma}\label{lem: calm}
Let the assumptions (A1-A6) and (B1-B4) be satisfied. Then the primal problem \eqref{eq: primal - NL  2} has an optimal solution $(\mu_{\upalpha_{\circ}},\upalpha_{\circ})$.
\end{lemma}

\begin{proof}
Using the existence and uniqueness of invariant measure discussed in Theorem \ref{thm existence inv meas} and Theorem \ref{thm closed extension}, the admissible set reduces to a singleton whenever $\upalpha\in\mathcal{A}$ is fixed, that is
\begin{equation*}
    \{q\in Q\,:\, G(q)\in K_{\upalpha}\}=\{\mu_{\upalpha}\},\quad \forall\,\upalpha\in \mathcal{A}.
\end{equation*}
Therefore, the problem \eqref{eq: primal - NL  2} can be equivalently expressed as
\begin{equation}
    \label{eq: primal - NL  sharp}
    \tag{$\mathfrak{P}_{_{_{\!\!\sharp}}}$}
    \min\limits_{\upalpha(\cdot)\in \mathcal{A}}\quad  F(\upalpha):=\langle f(\cdot\,,\upalpha(\cdot)),\mu_{\upalpha}\rangle.
\end{equation}
The objective function $F(\cdot)$ is the one introduced in \eqref{functional F}. Then with Proposition \ref{prop: continuity F}, this is a weak-$*$ continuous real-valued function on the weak-$*$ compact\footnote{This is a consequence of Banach–Alaoglu's theorem; see e.g. \cite[Theorem 3.16, p.66]{brezis2011functional}} subset $\mathcal{A}$ of $ L^{\infty}(\mathds{R}^{m})=(L^{1}(\mathds{R}^{m}))^{*}$. A classical result in optimization ensures that $F(\cdot)$ is bounded on $\mathcal{A}$ and achieves its minimum on $\mathcal{A}$; see e.g. \cite[Theorem 2, p.128]{luenberger1997optimization}.
\end{proof}

The latter existence result suggests a new description of the \textit{primal} problem \eqref{eq: primal - NL}. Given an optimal solution $(\mu_{\upalpha_{\circ}},\upalpha_{\circ})\in X\times \mathcal{A}$, the problem \eqref{eq: primal - NL} can be equivalently expressed as
\begin{equation}
    \label{eq: primal - NL  circ}
    \tag{$\mathfrak{P}_{_{_{\!\!\circ}}}$}
    \min\limits_{q\in Q}\quad  \langle f(\cdot\,,\upalpha_{\circ}(\cdot)),q\rangle,\quad \text{s.t.: }\; G(q)\in K_{\upalpha_{\circ}}.
\end{equation}
Indeed, solving \eqref{eq: primal - NL  circ} yields the unique invariant probability measure $\mu_{\upalpha_{\circ}}$. Yet, the advantage of this formulation is in the value function (see \S \ref{sec:duality theory}) $v(\cdot):Y=\mathds{R}\times X \to \mathds{R}$ 
\begin{equation*}
    v(y) = \min\limits_{q\in Q}\quad  \langle f(\cdot\,,\upalpha_{\circ}(\cdot)),q\rangle,\quad \text{s.t.: }\; G(q)+y\in K_{\upalpha_{\circ}}
\end{equation*}
whose argument is a perturbation in $q$ of the constraints, and not in $\upalpha$. Convexity (in fact, regularity) of $v(\cdot)$ is key for  strong duality to hold. This will be easy to handle as, we will later see, the problem  \eqref{eq: primal - NL  circ} is convex, unlike when formulated with \eqref{eq: primal - NL  2}. 

On the other hand, the formulation in \eqref{eq: primal - NL  2} will be needed in the subsequent section for the construction of the \textit{dual} problem, as we want to keep track of the \textit{minimization over $\upalpha$}; this is from where the Bellman Hamiltonian will appear. 

To sum up, we have three equivalent formulations for the same optimization problem, whose benefits/drawbacks are as follows:
\begin{description}
    \item[\eqref{eq: primal - NL  2}]{ is needed for the construction of the dual problem as it has the \textit{minimization over $\upalpha$} explicitly stated. But showing it is convex (in $(q,\upalpha)$) is hopeless.}
    \item[\eqref{eq: primal - NL  sharp}]{ is needed for proving the solvability of the primal problem, as it simplifies the constrained problem into an unconstrained one. But it will not be useful for duality as the desired dual variables measures sensitivity to constraints.}
    \item[\eqref{eq: primal - NL  circ}]{ is needed for showing that strong duality holds as it is a convex problem. But it requires existence of $\upalpha_{\circ}$ beforehand, as it is proved in Lemma \ref{lem: calm}.} 
\end{description}
The idea of using \eqref{eq: primal - NL  circ} may be reminiscent of the \textit{hidden convexity} in the celebrated Benamou-Brenier formulation of optimal transport \cite{benamou2000computational}. 

Finally, one may wonder if the set of dual solutions (obtained for the dual problem in the next section) may differ according to the formulation we adopt for the primal problem. In fact, the set of dual solutions is the same for any optimal solution of the primal problem; see \cite[Theorem 3.6, p.149]{bonnans2013perturbation}. Therefore, whether we fix one choice of an optimal solution as in \eqref{eq: primal - NL  circ} or we don't as in \eqref{eq: primal - NL 2}, the non-emptiness of the set of dual solutions will remain true and ultimately, this would yield the desired existence for our PDE problem.

\subsection{The dual problem}\label{sec: dual NL}

In order to deduce the corresponding \textit{dual} problem, we follow a parametric (conjugate) duality scheme as in \cite[\S 2.5.3, p. 107]{bonnans2013perturbation}. To this end, we adopt the formulation \eqref{eq: primal - NL  2} and embed the problem in a family of parameterized problems, where $y\in Y$ is the parameter vector and consider the function
\begin{equation*}
    \phi(q,y) = \min\limits_{\upalpha(\cdot)\in \mathcal{A}}\;\big\{ \, \langle f(\cdot\,,\upalpha(\cdot)),q\rangle + I_{K_{\upalpha}}(G(q) + y)\,\big\}. 
\end{equation*}
It is clear that when setting $y=0$, we recover the objective function in \eqref{eq: primal - NL  2}.

We also consider the following (Lagrangian) function, $L:X\times Y^{*}\times \mathcal{A} \to \mathds{R}$, analogue to \eqref{lagrangian} and s.t.
\begin{equation}\label{eq: Lagrangian function - NL}
    L(q,y^{*}, \upalpha) \coloneqq \langle f(\cdot\,,\upalpha(\cdot)),q\rangle + \langle y^{*}, G(q) \rangle_{Y^{*},Y}.
\end{equation}
Using the Legendre-Fenchel transform, we have
\begingroup
\allowdisplaybreaks
    \begin{align*}
    \phi^{*}(q^{*},y^{*}) & = \sup\limits_{q \in Q,y\in Y}\left\{\, \langle q^{*},q \rangle + \langle y^{*} , y \rangle - \phi(q,y)\,\right\}\\
    & = \sup\limits_{q \in Q,y\in Y}\left\{\, \langle q^{*},q \rangle + \langle y^{*} , y \rangle - \min\limits_{\upalpha(\cdot)\in \mathcal{A}}\;\{\;  \langle f(\cdot\,,\upalpha(\cdot)),q\rangle + I_{K_{\upalpha}}(G(q) + y)\}\,\right\}\\
    & = \sup\limits_{q \in Q,y\in Y}\left\{\max\limits_{\upalpha(\cdot)\in \mathcal{A}} \{ \langle q^{*},q \rangle + \langle y^{*} , y \rangle - \big(\langle f(\cdot\,,\upalpha(\cdot)),q\rangle + I_{K_{\upalpha}}(G(q) + y)\big)\,\}\right\}\\
    & = \max\limits_{\upalpha(\cdot)\in \mathcal{A}}\,\left\{\sup\limits_{q \in Q,y\in Y}\{\, \langle q^{*},q \rangle + \langle y^{*} , y \rangle - \big(\langle f(\cdot\,,\upalpha(\cdot)),q\rangle + I_{K_{\upalpha}}(G(q) + y)\big)\,\}\right\}\\
    & = \max\limits_{\upalpha(\cdot)\in \mathcal{A}}\left\{\,\sup\limits_{q \in Q}\,\{\, \langle q^{*},q \rangle - \langle f(\cdot\,,\upalpha(\cdot)),q\rangle - \langle y^{*},G(q) \rangle_{Y^{*},Y}\,\}\right. +\\
    &\quad \quad \quad \quad \quad \quad \quad \quad \quad \quad \quad \quad \left. + \sup\limits_{y\in Y} \{\,\langle y^{*}, G(q)+y \rangle - I_{K_{\upalpha}}(G(q) + y)\,\}\,\right\}\\
    & = \max\limits_{\upalpha(\cdot)\in \mathcal{A}}\left\{\,\sup\limits_{q \in Q}\,\{\, \langle q^{*},q \rangle - L(q,y^{*},\upalpha) + I_{K_{\upalpha}}^{*}(y^{*})\,\}\,\right\}\\
    & = \sup\limits_{q \in Q}\,\left\{ \langle q^{*},q \rangle + \max\limits_{\upalpha(\cdot)\in \mathcal{A}}\,\{ - L(q,y^{*},\upalpha) + I_{K_{\upalpha}}^{*}(y^{*})\}\,\right\}\\
    & = \sup\limits_{q \in Q}\,\left\{ \langle q^{*},q \rangle - \min\limits_{\upalpha(\cdot)\in \mathcal{A}} \,\{L(q,y^{*},\upalpha) - I_{K_{\upalpha}}^{*}(y^{*})\}\,\right\}
    \end{align*}
\endgroup
The \textit{dual} of the parametrized \textit{primal} problem is then obtained as
\begin{equation*}
    \max\limits_{y^{*}\in Y^{*}}\,\{ \langle y^{*}, y \rangle - \phi^{*}(0,y^{*})\,\}
\end{equation*}
which is
\begin{equation*}
    \max\limits_{y^{*}\in Y^{*}}\,\left\{ \langle y^{*}, y \rangle + \inf\limits_{q \in Q}\min\limits_{\upalpha(\cdot)\in \mathcal{A}} \{L(q,y^{*},\upalpha) - I_{K_{\upalpha}}^{*}(y^{*})\}\,\right\}.
\end{equation*}
Finally, the \textit{dual} problem associated to \eqref{eq: primal - NL  2} is obtained by setting $y=0$, that is
\begin{equation}
    \label{eq: dual - NL}
    \tag{$\mathfrak{D}$}
    \max\limits_{y^{*}\in Y^{*}}\,\left\{\inf\limits_{q \in Q}\min\limits_{\upalpha(\cdot)\in \mathcal{A}} \{L(q,y^{*},\upalpha) - I_{K_{\upalpha}}^{*}(y^{*})\}\,\right\}.
\end{equation}
In the next Lemma, we will make \eqref{eq: dual - NL} more explicit. But before we do so, we introduce an assumption that will play a crucial role in the validity of our method for solving the problem \eqref{cell prob - NL - 1}. Besides the standing assumptions (A1-A6) and (B1-B4) that we make, we denote again by $(\mathcal{L}_{\upalpha},D(\mathcal{L}_{\upalpha}))$ the closed extension of the diffusion operator as given by Theorem \ref{thm closed extension} and Theorem \ref{thm sobolev domain extension} and we assume the following holds true

\begin{description}
    \item[A*]{The domain $D(\mathcal{L}_{\upalpha})$ of the closed extension is nonempty and independent of $\upalpha$.}
\end{description}


This assumption means that there exists $\widetilde{\upalpha}(\cdot) \in \mathcal{A}$ such that for all $\upalpha(\cdot)\in \mathcal{A}$, one has $D(\mathcal{L}_{\upalpha}) = D(\mathcal{L}_{\widetilde{\upalpha}})$ , and  $\mathcal{L}_{\widetilde{\upalpha}}$ satisfies the standing assumptions, in particular it satisfies Theorem \ref{thm sobolev domain extension}. 
We will hereafter denote by $D(\mathcal{L}_{0})$ the latter domain.

\begin{remark}
A situation where one can check the validity of (A*) is in the convex case, that is for the family of second order elliptic operators
\begin{equation}
	\mathcal{L}_{\alpha}u(x) = \Delta u(x) - \nabla \Psi(x,\alpha)\cdot\nabla u(x),\quad x\in\mathds{R}^{n},
\end{equation}
where $\Psi(\,\cdot\,,\alpha)\in C^{2}(\mathds{R}^{m})$, $e^{-\Psi(\,\cdot\,,\alpha)}\in L^{1}(\mathds{R}^{m})$ and $\xi\cdot D^{2}\Psi(x,\alpha)\xi\geq 0$ for all $x,\xi\in\mathds{R}^{m}$, uniformly in $\alpha\in A$. In this case, one can get a complete characterization of its domain as shown in \cite[Thm. 8.4.2, p. 211]{{lorenzi2006analytical}} (see also \cite{da2004elliptic}) and that is
\begin{equation*}
	D(\mathcal{L}_{\alpha}) = \{u\in W^{2,2}(\mathds{R}^{m};\mu_{\alpha})\,:\, \nabla\Psi(\,\cdot\,,\alpha)\cdot \nabla u \in L^{2}(\mathds{R}^{m};\mu_{\alpha})\}.
\end{equation*}
Moreover, the invariant measure is $\mu_{\alpha}(x)\text{d}x = e^{-\Psi(x,\alpha)}\text{d}x$. As noted in \cite[Remark 3.3]{bardi2022deep}, if $\Psi(x,\alpha) = v_{0}(x) + v_{1}(x,\alpha) + v_{2}(\alpha)$, with $e^{-v_{0}(\cdot)}\in L^{1}(\mathds{R}^{m})$ and $\exists\,C>0$ constant such that $-C \leq \text{exp}(-v_{1}(x,\alpha)) \leq C$ for all $x\in\mathds{R}^{m}$, $\alpha\in A$, then $\mu_{\alpha}$ is equivalent to $e^{-v_{0}(\cdot)}\text{d}x$ independent of $\alpha$. Therefore, one gets $D(\mathcal{L}_{\alpha_{1}}) = D(\mathcal{L}_{\alpha_{1}}) \subset W^{p,2}_{loc}(\mathds{R}^{m})$ for any $\alpha_{1},\alpha_{2}\in A, p\in [1,+\infty)$ and hence (A*) is satisfied. \\
An example of such elliptic operators with unbounded coefficients is the Ornstein-Uhlenbeck that we define by
\begin{equation*}
	\mathcal{L}_{\alpha} u(x) = \frac{1}{2}\text{trace}(Q(\alpha)D^{2}u(x)) + B(\alpha)x \cdot \nabla u(x),\quad x\in\mathds{R}^{m}
\end{equation*}
where $Q(\alpha)$ and $B(\alpha)$ are $m\times m$ matrices independent of $x$, with $Q$ strictly positive definite and $B\neq 0$ with a spectrum contained in the left halfplane, for all $\alpha$.  In this case, one has an explicit formula for the invariant measure and for the domain $D(\mathcal{L}_{\alpha})$, see \cite[Chapter 10]{lorenzi2016analytical}. Hence, one could check the validity of (A*) with a similar argument as before. \\
Another situation where assumption (A*) is satisfied is in the case where all the data of our problem $a,b,f$ are smooth, then $D(\mathcal{L}_{0})$ can be chosen as the subset of functions in $C^{2}(\mathds{R}^{m})$ satisfying some polynomial growth. This is done in \cite[\S III.6, p. 130]{fleming2006controlled} in the context of stochastic control. \\
In general, (A*) can be satisfied when we have equality in \eqref{eq: domain extension}. This is true if the formal adjoint to $\mathcal{L}_{\alpha}$ is essentially m-dissipative on $L ^{r'}(\mathds{R}^{m};\mu_{\alpha})$, where $1/r + 1/r' = 1$, see \cite[Theorem 2.8(ii)]{bogachev2002uniqueness}. See \cite{bogachev2002uniqueness} for a definition of ``formal adjoint operator'' and ``essential m-dissipativity". A sufficient condition for essential m-dissipativity, and hence for equality in \eqref{eq: domain extension}, is given in \cite[Theorem 3.1(iii)]{bogachev2002uniqueness}.
\end{remark}

\hfill

The following result provides a less abstract formulation of the dual problem and unveils the presence of the PDE in our optimization framework. 

\begin{lemma}\label{main 0 - NL}
The problem \eqref{eq: dual - NL} is equivalent to
\begin{equation}
    \label{eq: dual - NL 5}
    \tag{$\mathfrak{D}$}
    \max\limits_{\substack{c\in\mathds{R}\\u\in \mathcal{X}}}\;\left\{\, c, \;\text{s.t.: }\; c-H(x,\nabla u, D^{2}u) \leq 0,\; a.e. \text{ in } \mathds{R}^{m}\,\right\}
\end{equation}
where $H(x,\nabla u(x), D^{2}u(x)) = \min\limits_{\alpha\in A}\{\, -\mathcal{L}_{\alpha}u(x) + f(x,\alpha) \,\}$ and $\mathcal{X}$ is such that
\begin{equation}\label{eq: functional space X - NL}
    \mathcal{X} = D(\mathcal{L}_{0})\cap\{u:\mathds{R}^{m}\to \mathds{R}, \textit{Borel-meas.}\;|\; \exists\;C>0,\; |u(x)| \leq C(1+|x|^{\kappa})\}
\end{equation}
with $\kappa = d+1-\theta$,  that is, the two optimization problems have the same set of optimal solutions and the same optimal value.
\end{lemma}
\begin{remark}
(A*) together with Theorem \ref{thm sobolev domain extension} ensure that $D(\mathcal{L}_{0}) \subset W^{r,2}_\text{loc}(\mathds{R}^{m})$. 
\end{remark}

\begin{proof}
It is wellknown that the conjugate of the indicator function is the \textit{support function} (see, e.g., \cite[Example 2.115, p. 80]{bonnans2013perturbation}), that is, 
\begin{equation}\label{def support function}
    \begin{aligned}
    I^{*}_{K_{\alpha}}(y^{*})
    & = \sup\limits_{z\in K_{\alpha}}\; \langle y^{*},z\rangle  = \left\{
    \begin{aligned}
    0 ,\quad   &\text{ if }\; y^{*}\in (K_{\upalpha})^{-}&\\
    +\infty,\quad   &\text{ otherwise }&
    \end{aligned}\right.
    \end{aligned}
\end{equation}
Recalling the definition $K_{\upalpha} = \{0\}\times \text{Ker}(\mathcal{L}_{\upalpha})$, we have
\begin{equation*}
    \begin{aligned}
    y^{*}\in (K_{\upalpha})^{-} & \Leftrightarrow\, (c,\omega) \in \bigg(\{0\}\times \text{Ker}(\mathcal{L}_{\upalpha})\bigg)^{-}\\
    & \Leftrightarrow\, (c,\omega) \in \mathds{R}\times (\text{Ker}(\mathcal{L}_{\upalpha}))^{\bot}\\
    & \Leftrightarrow\, (c,\omega) \in \mathds{R}\times \text{cl}(\text{range}(\mathcal{L}_{\upalpha}))
    \end{aligned}
\end{equation*}
Since we are working with $\mathcal{L}_{\upalpha}$ in its closed extension, we have
\begin{equation*}
    \begin{aligned}
    \omega\in \text{cl}(\text{range}(\mathcal{L}_{\upalpha})) & \Leftrightarrow\, \exists\;u\in D(\mathcal{L}_{\upalpha}),\;\text{s.t. }\; \omega = -\mathcal{L}_{\upalpha}u\\
    & \Leftrightarrow\, \exists\;u\in D(\mathcal{L}_{0}),\;\text{s.t. }\; \omega = -\mathcal{L}_{\upalpha}u
    \end{aligned}
\end{equation*}
where the last equivalence is obtained thanks to the assumption (A*) which guarantees that $D(\mathcal{L}_{\upalpha}) = D(\mathcal{L}_{0})$ for all $\upalpha(\cdot) \in \mathcal{A}$. Note however that $\omega$ still depends on $\alpha$ through its definition as $\omega = -\mathcal{L}_{\upalpha}u$.  
Our \textit{dual} problem is now
\begin{equation}
    \label{eq: dual - NL 2}
    \tag{$\mathfrak{D}$}
    \max\limits_{y^{*}\in Y^{*}}\,\inf\limits_{q \in Q}\min\limits_{\upalpha(\cdot)\in \mathcal{A}} \{L(q,y^{*},\upalpha) \;\;\text{ s.t. }\, y^{*}=(c,-\mathcal{L}_{\upalpha}u)\;\text{and}\, (c,u)\in \mathds{R}\times D(\mathcal{L}_{0})\},
\end{equation}
and the new variables on which we perform the maximization are now $(c,u)$ and they belong to  $\mathds{R}\times D(\mathcal{L}_{0})$. The latter being independent of $\upalpha(\cdot)$, we can isolate it from the minimization over $\upalpha$ and write it as a subscript of the maximization over $(c,u)$. Then the \textit{dual} problem becomes
\begin{equation}
    \label{eq: dual - NL 3}
    \tag{$\mathfrak{D}$}
    \max\limits_{\substack{c\in\mathds{R}\\u\in D(\mathcal{L}_{0})}}\,\inf\limits_{q \in Q}\min\limits_{\upalpha(\cdot)\in \mathcal{A}} \{L(q,y^{*},\upalpha),\;\;\text{ s.t. }\, y^{*}=(c,-\mathcal{L}_{\upalpha}u)\,\}.
\end{equation}
Recalling the definition \eqref{eq: Lagrangian function - NL} of $L$ and the notations introduced earlier, we have
\begin{equation*}
    \begin{aligned}
    L(q, y^{*},\upalpha) & = \langle f(\cdot\,,\upalpha(\cdot)),q\rangle + \langle y^{*}, G(q) \rangle_{Y^{*},Y}\\
    & = \langle f(\cdot\,,\upalpha(\cdot)) ,q\rangle + c(1-\langle 1, q \rangle) + \langle -\mathcal{L}_{\upalpha}u(\cdot),q \rangle\\
    & = c + \langle f(\cdot\,,\upalpha(\cdot)) - \mathcal{L}_{\upalpha}u(\cdot) - c , q \rangle.
    \end{aligned}
\end{equation*}
Hence we have, using the \textit{exchange property} in Proposition \ref{prop: exchange prop},
\begin{equation*}
    \begin{aligned}
    & \min\limits_{\upalpha(\cdot)\in \mathcal{A}} \{L(q,y^{*},\upalpha),\;\text{ s.t. } y^{*}=(c,-\mathcal{L}_{\upalpha}u)\,\}  = c + \min\limits_{\upalpha(\cdot)\in \mathcal{A}}\left\{\, \langle f(\cdot\,,\upalpha(\cdot)) - \mathcal{L}_{\upalpha}u(\cdot) - c , q \rangle\right\}\\
    & \quad \quad \quad \quad \quad = c + \langle \min\limits_{\alpha\in A}\{f(\cdot\,,\alpha) - \mathcal{L}_{\alpha}u(\cdot)\} - c , q \rangle = c + \langle H(\cdot\,,\nabla u, D^{2}u) - c , q \rangle
    \end{aligned}
\end{equation*}
and the \textit{dual} problem is
\begin{equation}
    \label{eq: dual - NL 4}
    \tag{$\mathfrak{D}$}
    \max\limits_{\substack{c\in\mathds{R}\\u\in D(\mathcal{L}_{0})}}\left\{\; c + \;\inf\limits_{q \in Q}\langle H(\cdot\,,\nabla u, D^{2}u) - c , q \rangle\;\right\}.
\end{equation}
Noting that $-\inf\limits_{q \in Q}\langle H(\cdot\,,\nabla u, D^{2}u) - c , q \rangle = \sup\limits_{q \in Q}\langle c - H(\cdot\,,\nabla u, D^{2}u) , q \rangle$ is the support function \eqref{def support function} which is $0$ if $\langle c - H(\cdot\,,\nabla u, D^{2}u) , q \rangle \leq 0$ for all $q\in Q$ and $+\infty$ otherwise. But since $Q$ is made of non-negative measures with finite moment of order $d$, we firstly need that $c - H(x,\nabla u, D^{2}u) \leq 0$ a.e. on the support\footnote{The support of a measure $\mu$ is  $\text{spt}(\mu):=\{z\in \mathds{R}^{m} \,:\, \mu(U)>0 \, \text{ for each neighborhood } U \text{ of } z\}$.} of each $q\in Q$, hence in $\mathds{R}^{m}$, and secondly we need $u$ to have a polynomial growth of order at most $\kappa=d+1-\theta$. Indeed, since $f$ has a growth of order at most $d$ (by assumption (A5)) and $c$ is a constant, we need $\mathcal{L}_{\alpha}u$ to satisfy this same growth condition. By assumption (A3), the matrix function $a$ is uniformly bounded, and by assumption (A6) the drift vector field has a polynomial growth of order $\theta$. Hence, setting $\kappa$ as the polynomial growth of $u$,  it necessarily satisfies $\kappa -1 + \theta \leq d$ where $\kappa-1$ corresponds to the growth of $\nabla u$.
So a sufficient condition to have $\mathcal{L}_{\alpha}u$ (and hence $H(x,\nabla u,D^{2}u)$) with a polynomial growth of order at most $d$ is to have $u$ satisfying a polynomial growth of order at most $\kappa = d+1-\theta$ (note that $\kappa\geq 1$ since $\theta\in [0,d]$).\\
The \textit{dual} problem is finally
\begin{equation}
    \tag{$\mathfrak{D}$}
    \max\limits_{\substack{c\in\mathds{R}\\u\in \mathcal{X}}}\left\{\; c\,, \;\text{s.t.: }\; c-H(x,\nabla u, D^{2}u) \leq 0,\; a.e. \text{ in } \mathds{R}^{m}\;\right\}
\end{equation}
and the functional space $\mathcal{X}$ is now
\begin{equation*}
    \mathcal{X} = D(\mathcal{L}_{0})\cap\{u:\mathds{R}^{m}\to \mathds{R}, \textit{Borel-meas.}\;|\; \exists\;C>0,\; |u(x)| \leq C(1+|x|^{\kappa})\}
\end{equation*}
where $\kappa=d+1-\theta$, which then concludes the proof.
\end{proof}


In the case where the Hamiltonian is given by
\begin{equation*}
    H(x,\nabla u(x), D^{2}u(x)) = \max\limits_{\alpha\in A}\{\, -\mathcal{L}_{\alpha}u(x) + f(x,\alpha) \,\},
\end{equation*}
the same proof as before can again be conducted, with minor modification in the duality procedure. We refer to \cite{kouhkouhPhD} for further details. 

\subsection{The ergodic HJB equation}

\subsubsection{The optimality conditions}\label{subsec: opt cond}

Recall the Hamiltonian
\begin{equation*}
    H(x,\nabla u(x), D^{2}u(x)) = \min\limits_{\alpha\in A}\{\, -\mathcal{L}_{\alpha}u(x) + f(x,\alpha) \,\}.
\end{equation*}

We check that the optimality conditions as stated in \S\ref{sec:duality theory}, in particular \eqref{optimality conditions - 2} and \eqref{equiv cond normal cone}, still hold in our framework. In order to do so, we start from the \textit{duality gap} (or \textit{duality inequality}) which states that the value of the \textit{dual} problem \eqref{eq: dual - NL} is less or equal than the value of the \textit{primal} problem \eqref{eq: primal - NL  2}. Recalling the definition \eqref{eq: Lagrangian function - NL} of the Lagrangian function $L$ and the value of the \textit{dual} problem being less or equal the value of the \textit{primal} problem (see \S \ref{sec:duality theory}), we have
\begin{equation*}
\begin{aligned}
    &\max\limits_{y^{*}\in Y^{*}}\min\limits_{q\in Q}\min\limits_{\upalpha(\cdot)\in \mathcal{A}}\{L(q,y^{*},\upalpha)-I_{K_{\upalpha}}^{*}(y^{*})\}\\
    &\quad\quad\quad \leq \min\limits_{q\in Q}\min\limits_{\upalpha(\cdot)\in \mathcal{A}}\{\langle f(\cdot\,,\upalpha(\cdot)),q\rangle + I_{K_{\upalpha}}(G(q))\}\\
    &\quad\quad\quad = \min\limits_{q\in Q}\min\limits_{\upalpha(\cdot)\in \mathcal{A}}\{ L(q,y^{*},\upalpha) + I_{K_{\upalpha}}(G(q)) - \langle y^{*},G(q) \rangle_{Y^{*},Y}\},\;\forall\,y^{*}\in Y^{*}.
\end{aligned}
\end{equation*}
Let us denote by $(q_{\circ},\upalpha_{\circ})$ an optimal solution in the \textit{primal} problem \eqref{eq: primal - NL  2} and by $y^{*}_{\circ}$ an optimal solution in the \textit{dual} problem \eqref{eq: dual - NL}. We have
\begin{equation}\label{eq:gap duality for opt cond}
\begin{aligned}
    \min\limits_{q\in Q}\min\limits_{\upalpha(\cdot)\in \mathcal{A}}\{L(q,y^{*}_{\circ},\upalpha)-I_{K_{\upalpha}}^{*}(y^{*}_{\circ})\} & \leq  L(q_{\circ},y^{*}_{\circ},\upalpha_{\circ}) + I_{K_{\upalpha_{\circ}}}(G(q_{\circ})) - \langle y^{*}_{\circ},G(q_{\circ}) \rangle_{Y^{*},Y}\\
    & = \langle f(\cdot\,, \upalpha_{\circ}(\cdot)), q_{\circ} \rangle + I_{K_{\upalpha_{\circ}}}(G(q_{\circ})).
\end{aligned}
\end{equation}

The optimality conditions are  obtained when we reach equality in the above inequality. We can then characterize the optimal \textit{primal} and \textit{dual} solutions and provide a no-\textit{duality gap} condition. Suppose the left hand side minimization in the above inequality is reached in the pair of optimal solutions $(q_{\circ},\upalpha_{\circ})$. 
Therefore, the latter inequality reduces to
\begin{equation*}
0 \leq I_{K_{\upalpha_{\circ}}}^{*}(y^{*}_{\circ}) +  I_{K_{\upalpha_{\circ}}}(G(q_{\circ})) - \langle y^{*}_{\circ},G(q_{\circ}) \rangle_{Y^{*},Y}.
\end{equation*}
This is the Young-Fenchel inequality, and equality holds if and only if we have 
\begin{equation}\label{eq: opt cond NL - 3}
    y^{*}_{\circ}\in  \partial I_{K_{\upalpha_{\circ}}}(G(q_{\circ})) = N_{K_{\upalpha_{\circ}}}(G(q_{\circ})).
\end{equation}
Since $K_{\upalpha_{\circ}}$ is a convex cone, then $y^{*}_{\circ}\in N_{K_{\upalpha_{\circ}}}(G(q_{\circ}))$ is equivalent to 
\begin{equation}
    G(q_{\circ})\in K_{\upalpha_{\circ}},\quad y^{*}_{\circ}\in (K_{\upalpha_{\circ}})^{-}\; \text{ and }\; \langle y^{*}_{\circ},G(q_{\circ}) \rangle_{Y^{*},Y} = 0.
\end{equation}
Moreover, and recalling the definition \eqref{def support function}, we have $I_{K_{\upalpha_{\circ}}}^{*}(y^{*}_{\circ})=0$  when $y^{*}_{\circ}\in (K_{\upalpha_{\circ}})^{-}$. So, going back to the inequality in \eqref{eq:gap duality for opt cond} which we are now supposing to be an equality (\textit{no-duality gap}), we have 
\begin{equation*}
\begin{aligned}
    \min\limits_{q\in Q}\min\limits_{\upalpha(\cdot)\in \mathcal{A}}\{L(q,y^{*}_{\circ},\upalpha)-I_{K_{\upalpha}}^{*}(y^{*}_{\circ})\} =  L(q_{\circ},y^{*}_{\circ},\upalpha_{\circ}) - I_{K_{\upalpha_{\circ}}}^{*}(y^{*}_{\circ}) =  L(q_{\circ},y^{*}_{\circ},\upalpha_{\circ}).
\end{aligned}
\end{equation*}
Recalling \eqref{def support function}, we have $\min\limits_{q\in Q}\min\limits_{\upalpha(\cdot)\in \mathcal{A}}\{L(q,y^{*}_{\circ},\upalpha)-I_{K_{\upalpha}}^{*}(y^{*}_{\circ})\} \leq  \min\limits_{q\in Q}\min\limits_{\upalpha(\cdot)\in \mathcal{A}} \; L(q,y^{*}_{\circ},\upalpha)$ 
which finally yields, together with the previous equality,
\begin{equation*}
    L(q_{\circ},y^{*}_{\circ},\upalpha_{\circ}) \leq \min\limits_{q\in Q}\min\limits_{\upalpha(\cdot)\in \mathcal{A}} \; L(q,y^{*}_{\circ},\upalpha).
\end{equation*}

To sum up, we have the following sufficient optimality conditions that are indeed analogue to \eqref{optimality conditions - 2}, and which also guarantee the absence of the \textit{duality gap}
\begin{equation} \label{optimality conditions - NL}
\left\{
    \begin{aligned}
    & (q_{\circ},\upalpha_{\circ}) \in \argmin\limits_{q\in Q, \upalpha(\cdot)\in \mathcal{A}}\; L(q,y^{*}_{\circ},\upalpha)\\
    & G(q_{\circ})\in K_{\upalpha_{\circ}},\quad y^{*}_{\circ}\in (K_{\upalpha_{\circ}})^{-}\; \text{ and }\; \langle y^{*}_{\circ},G(q_{\circ}) \rangle_{Y^{*},Y} = 0.
    \end{aligned}
\right.
\end{equation}

\subsubsection{The main result}

We are now ready to state and prove the existence and uniqueness result for a solution to the ergodic HJB equation as given in our initial problem \eqref{cell prob - NL - 1}, assuming (A1-A6), (B1-B4) and (A*) hold true.

\begin{theorem}\label{thm: main NL}
There exists a pair $(c,u(\cdot))\in \mathds{R}\times W^{r,2}_{\text{loc}}(\mathds{R}^{m})$ for any $r\in[1,+\infty)$, such that $|u(x)|\leq K(1+|x|^{\kappa})$ where $\kappa=d-1+\theta$ and $K>0$ a constant, solution to 
\begin{equation*}
    H(x,\nabla u(x),D^{2}u(x)) = c,\quad \text{a.e. in }\mathds{R}^{m}
\end{equation*}
where $H(x,p,P) = \min\limits_{\alpha\in A} \{\, -b(x,\alpha)\cdot p - \text{trace}(a(x,\alpha)P) + f(x,\alpha) \,\}$. \\ 
Moreover, the latter constant $c$ is given by $c=\langle f(\cdot\,,\upalpha(\cdot))\,,\, \mu_{\upalpha} \rangle$ where 
\begin{equation*}
    \upalpha(x)\in\argmin\limits_{\alpha\in A} \{\, -\mathcal{L}_{\alpha}u(x) + f(x,\alpha) \,\},\quad \text{a.e. in }\mathds{R}^{m}
\end{equation*}
and $\mu_{\upalpha}$ is the unique invariant probability measure associated to $\mathcal{L}^{*}_{\upalpha}$.

When $r>\frac{m}{2}$,  $u(\cdot)$ is continuous and pointwise twice differentiable almost everywhere. If, for the latter specific constant $c$, we assume moreover that the vector field $b$ is locally Lipschitz continuous in $x$ uniformly in $\alpha$, and $\theta=1$ in (A6), then $u(\cdot)$ with such a polynomial growth is unique in any $W^{r,2}_{\text{loc}}(\mathds{R}^{m})$, $r>\frac{m}{2}$, in the sense: if $(c,u_{1}(\cdot))$ and $(c,u_{2}(\cdot))$ are two solutions, then $u_{1}(\cdot)-u_{2}(\cdot)\equiv\,\text{constant}$.
\end{theorem}

\begin{remark}\label{rmk:regularity hjb}~~\\
\textbullet\quad The HJB equation is solved on $\text{spt}(\mu)$ (the support of the unique invariant measure $\mu$). But thanks to Theorem \ref{thm regularity meas}, we have $\mu \ll dx$ and $\text{spt}(\mu)=\mathds{R}^m$. In fact, a more general statement of our problem would be
\begin{equation*} 
\begin{aligned}
    &\text{Find (the largest) }\, \Omega\subset\mathds{R}^{m} \text{open} \text{ and } (c,u(\cdot))\in\mathds{R}\times \mathcal{X}(\Omega)\, \text{ s.t.: }\\
 &\quad \quad \quad \quad \quad H(x,\nabla u(x),D^{2}u(x)) = c,\;\text{ in }\Omega.
\end{aligned}
\end{equation*}
Then our result gives $(c,u)$ as in Theorem \ref{thm: main NL} with  $\Omega=\text{spt}(\mu)$, but $\text{spt}(\mu)=\mathds{R}^m$.\\
\textbullet\quad We note that $u(\cdot)$ is a strong $L$-viscosity solution (see  \cite{caffarelli1996viscosity, crandall1996equivalence}), which is as expected as when we consider $C$-viscosity solutions for the continuous case. Recall that in our setting, the vector field $b$ and the function $f$ are assumed to be measurable in $x$.\\
\textbullet\quad If $2r>m$, then a classical embedding theorem (see, e.g., \cite[Chapter 5]{adams1975sobolev}) states that $W^{r,2}(\Omega)\subset C(\Omega)$ for any $\Omega$ bounded subset of $\mathds{R}^m$ satisfying the \textit{cone property}. 
By using smooth cut-off functions $\zeta\in C^{\infty}_{0}(\mathds{R}^{m})$ with a support $U$ bounded subset of $\mathds{R}^m$, we have $\zeta\,u \in W^{r,2}(U)\subset C(U)$. We conclude that for any $r>\frac{m}{2}$, the solution $u(\cdot) \in W^{r,2}_{\text{loc}}(\mathds{R}^m)$ is a continuous function. Note also that the range $2r>m$ is the one where $W^{r,2}_{\text{loc}}(\mathds{R}^{m})$ functions are not only continuous but also pointwise twice differentiable almost everywhere (see, e.g., \cite[Appendix C]{caffarelli1996viscosity}). And in this case, $u(\cdot)$ shall be a $C$-viscosity solution.
\end{remark}

\begin{proof}[Proof of Theorem \ref{thm: main NL}]
Let $\upalpha_{\circ}\in \mathcal{A}$ be an optimal solution for \eqref{eq: primal - NL  sharp} in the proof of Lemma \ref{lem: calm}, and let us consider the \textit{primal} problem formulated as 
\begin{equation}
    \label{eq: primal - NL  circ proof}
    \tag{$\mathfrak{P}_{_{_{\!\!\circ}}}$}
    \min\limits_{q\in Q}\quad  \langle f(\cdot\,,\upalpha_{\circ}(\cdot)),q\rangle,\quad \text{s.t.: }\; G(q)\in K_{\upalpha_{\circ}}.
\end{equation}
We recall the dual problem from Lemma \ref{main 0 - NL}
\begin{equation}
    \label{eq: dual - NL - proof}
    \tag{$\mathfrak{D}$}
    \max\limits_{\substack{c\in\mathds{R}\\u\in \mathcal{X}}}\left\{\, c, \;\text{s.t.: }\; c-H(x,\nabla u, D^{2}u) \leq 0,\; a.e. \text{ in } \mathds{R}^{m}\,\right\}.
\end{equation}

\textit{Step 1. (On the optimization problems)}\\
We need to check if the assumptions of Theorem \ref{thm: duality} are satisfied by \eqref{eq: primal - NL  circ proof}. The objective function $q\mapsto \langle f(\cdot\,\upalpha_{\circ}(\cdot)),q \rangle$ is linear hence convex and continuous, the set $Q=\mathcal{M}^{+}_{d}(\mathds{R}^{m})$ is clearly convex and close, the function $G(q) = (G_{1}(q),G_{2}(q))$, with $G_{1}(q) = 1-\langle 1,q\rangle$ and $G_{2}(q)=q$, is continuously differentiable and convex w.r.t. the set $-K$ (this is easy to check as $G$ is affine). The last assumption we need is \eqref{eq: interior} which is in our situation equivalent to \eqref{eq: interior 2} as shown by Proposition \ref{prop: interior}. Let $q$ be a feasible point and recall the notation in \S \ref{sec: primal NL}. Using the results in \S \ref{sec:extension diff}, in particular Theorem \ref{thm regularity meas} and Theorem \ref{thm existence inv meas}, we have $K_{2}(\upalpha_{\circ})= \text{Ker}(\mathcal{L}^{*}_{\upalpha_{\circ}}) = \{h\,:\, h= \lambda \mu_{\upalpha_{\circ}},\,\lambda \geq 0\}$. We can then write 
\begin{equation*}
\begin{aligned}
    & G_{1}(q) +DG_{1}(q)[K_{2} - q] - K_{1} = 1- \langle 1, q\rangle + \{\,-\langle 1, h - q \rangle \;:\; \forall\, h\in K_{2}\}\\
    & \quad \quad \quad = 1 - \{\,\lambda\,\langle 1, \mu_{\upalpha_{\circ}}  \rangle \;:\; \forall\,\lambda\geq 0\} = (-\infty,1]
\end{aligned}
\end{equation*}
where in the last equality we used the fact that $\mu_{\upalpha_{\circ}}$ is a probability measure hence $\langle 1,\mu_{\upalpha_{\circ}}\rangle = 1$. Therefore $0\in \text{int}\{G_{1}(q) +DG_{1}(q)[K_{2} - q] - K_{1}\}$ and we can apply Theorem \ref{thm: duality}. Moreover, Lemma \ref{lem: calm} ensures that the primal problem has a finite value (because it has a solution). Thus we have $(i)$ no duality gap between the primal and dual problem and, $(ii)$ existence of a nonempty set of solutions to the dual problem\footnote{Theorem \ref{thm: duality} tells us more: the optimal set of solutions of the dual problem is nonempty, convex, bounded and weak-$*$ compact subset of $Y^{*}$.}.

Let us now denote by $(q_{\circ},\upalpha_{\circ})\in Q\times \mathcal{A}$ and $(c_{\circ},u_{\circ})\in \mathds{R}\times \mathcal{X}$ optimal solutions of \eqref{eq: primal - NL  circ proof} and \eqref{eq: dual - NL - proof} respectively. With Theorem \ref{thm: opt cond}, they satisfy the optimality conditions \eqref{optimality conditions - NL} with $y^{*}_{\circ}\coloneqq (c_{\circ},-\mathcal{L}_{\upalpha_{\circ}}u_{\circ})$ (see the proof of Lemma \ref{main 0 - NL}).

\textit{Step 2. (On the PDE problem)}\\
We need to \textit{translate} the optimality conditions \eqref{optimality conditions - NL} into a PDE. \\
We start from the  no-\textit{duality gap}: it yields 
\begin{equation}\label{eq: no gap proof}
    c_{\circ} = \langle f(\cdot\,, \upalpha_{\circ}(\cdot)), q_{\circ} \rangle.
\end{equation}
Then, the condition $\langle y^{*}_{\circ},G(q_{\circ}) \rangle_{Y^{*},Y} = 0$ is just $c_{\circ}(1-\langle 1, q_{\circ}\rangle) + \langle -\mathcal{L}_{\upalpha_{\circ}}u_{\circ}(\cdot) , q_{\circ} \rangle = 0$ which together with \eqref{eq: no gap proof} becomes
\begin{equation}\label{c circ}
\begin{aligned}
    & \langle f(\cdot\,, \upalpha_{\circ}(\cdot)), q_{\circ} \rangle-\langle c_{\circ}, q_{\circ}\rangle + \langle -\mathcal{L}_{\upalpha_{\circ}}u_{\circ}(\cdot) \,,\, q_{\circ} \rangle = 0\\
     & \Leftrightarrow\quad \langle\,  -\mathcal{L}_{\upalpha_{\circ}}u_{\circ}(\cdot) + f(\cdot\,, \upalpha_{\circ}(\cdot))\, , \,q_{\circ} \rangle = \langle c_{\circ},q_{\circ}\rangle =  c_{\circ}.
\end{aligned}
\end{equation}
Note that $q_{\circ}$ here is what we denoted by $\mu_{\upalpha_{\circ}}$, i.e. the unique invariant probability measure associated to $\mathcal{L}^{*}_{\upalpha_{\circ}}$. \\
On the other hand, $(c_{\circ},u_{\circ})$ solves \eqref{eq: dual - NL - proof}, in particular the constraint is satisfied, that is
\begin{equation}\label{constant sign  H}
    c_{\circ} - H(x,\nabla u_{\circ}(x),D^{2}u_{\circ}(x)) \leq 0,\quad \text{a.e. in }\mathds{R}^{m}
\end{equation}
where we recall $H(x,\nabla u_{\circ}(x),D^{2}u_{\circ}(x)) := \min\limits_{\alpha \in A}\{-\mathcal{L}_{\alpha}u_{\circ}(x) + f(x,\alpha)\}$, thus 
\begin{equation*}
    \begin{aligned}
        c_{\circ}  \leq H(x,\nabla u_{\circ}(x),D^{2}u_{\circ}(x)) \leq -\mathcal{L}_{\alpha}u_{\circ}(x) + f(x,\alpha),\quad \forall \alpha\in A, 
    \end{aligned}
\end{equation*}
in particular
\begin{equation}\label{sign}
    c_{\circ} \leq H(x,\nabla u_{\circ}(x),D^{2}u_{\circ}(x)) \leq  - \mathcal{L}_{\upalpha_{\circ}}u_{\circ}(x) + f(x,\upalpha_{\circ}(x)),\quad \text{a.e. in }\mathds{R}^{m}.
\end{equation}
Integrating these inequalities w.r.t. $q_{\circ}$ yields
\begin{equation*}
    c_{\circ} \leq \langle H(\cdot\,,\nabla u_{\circ}(\cdot),D^{2}u_{\circ}(\cdot)) , q_{\circ} \rangle \leq \langle\, -\mathcal{L}_{\upalpha_{\circ}}u_{\circ}(\cdot) + f(\cdot\,, \upalpha_{\circ}(\cdot))\, , \,q_{\circ} \rangle.
\end{equation*}
Using \eqref{c circ}, the latter are in fact equalities. Therefore
\begin{equation*}
    \langle H(\cdot\,,\nabla u_{\circ}(\cdot),D^{2}u_{\circ}(\cdot))-\big(\, -\mathcal{L}_{\upalpha_{\circ}}u_{\circ}(\cdot) + f(\cdot\,, \upalpha_{\circ}(\cdot))\big)\, , \,q_{\circ} \rangle = 0
\end{equation*}
and with \eqref{sign}, the function $H(\cdot\,,\nabla u_{\circ}(\cdot),D^{2}u_{\circ}(\cdot))-\big(\, -\mathcal{L}_{\upalpha_{\circ}}u_{\circ}(\cdot) + f(\cdot\,, \upalpha_{\circ}(\cdot))\big)\leq 0$ does not change sign. This implies
\begin{equation}\label{H=L}
    H(x,\nabla u_{\circ}(x),D^{2}u_{\circ}(x)) = -\mathcal{L}_{\upalpha_{\circ}}u_{\circ}(x) + f(x,\upalpha_{\circ}(x)),\quad \text{for } q_{\circ}\text{-a.e. } x\in \text{spt}(q_{\circ}),
\end{equation}
where we recall $\text{spt}(q_{\circ})$ is the support of $q_{\circ}$. In other words, we have
\begin{gather*}
    H(x,\nabla u_{\circ}(x),D^{2}u_{\circ}(x)) = c_{\circ},\quad q_{\circ}\text{-a.e. } x\in \text{spt}(q_{\circ}),\\
    \upalpha_{\circ}(x) \in \argmin\limits_{\alpha \in A}\,\big\{\, -\mathcal{L}_{\alpha}u_{\circ}(x) + f(x,\alpha)\;\big\},\quad q_{\circ}\text{-a.e. } x\in \text{spt}(q_{\circ}),
\end{gather*}
where the first statement is a consequence of \eqref{H=L}, the constant sign from \eqref{constant sign  H} and of \eqref{c circ}, whereas the second statement is based on the definition of $H$ and \eqref{H=L}. 
But $q_{\circ}$ is absolutely continuous with respect to Lebesgue measure and is supported in the whole $\mathds{R}^{m}$ (see Theorem \ref{thm regularity meas}), hence the results almost everywhere in $\mathds{R}^{m}$, and \vspace*{-1.5em}
\begin{center}
    $(c_{\circ},u_{\circ})$ solves \eqref{cell prob - NL - 1} where $\mathcal{X}$ is as in \eqref{eq: functional space X - NL}.\vspace*{-0.5em}
\end{center}

\textit{Step 3. (Uniqueness of $u_{\circ}(\cdot)$)} \\
To prove that $u_{\circ}(\cdot)$ is unique, we need to assume in addition that  $b$ is locally Lipschitz continuous with at most a linear growth, i.e. $\theta = 1$ and hence $\kappa = d$. This setting will allow us to apply the Liouville type result in \cite{bardi2016liouville}. \\
Suppose $(c_{\circ},u_{1}(\cdot)), (c_{\circ},u_{2}(\cdot))$ are two solutions such that $u_{1},u_{2}\in W^{r,2}_{\text{loc}}$ for $r>\frac{m}{2}$, and with a polynomial growth of order at most $d$. Then we have, using the inequality ``$\min(A-B)\leq \min(A) - \min(B)$"
\begin{equation*}
    \min\limits_{\alpha\in A}\{\,-\mathcal{L}_{\alpha}(u_{1}-u_{2})\,\} \leq \min\limits_{\alpha\in A}\{-\mathcal{L}_{\alpha}u_{1} + f(\cdot\,,\alpha)\} - \min\limits_{\alpha\in A}\{-\mathcal{L}_{\alpha}u_{2} + f(\cdot\,,\alpha)\}  = 0.
\end{equation*}
Note also that when $r>\frac{m}{2}$, $W^{r,2}_{\text{loc}}$ functions are continuous and pointwise twice differentiable almost everywhere (see the last point in Remark \ref{rmk:regularity hjb}). So $v\coloneqq u_{1}-u_{2}$ is a continuous viscosity sub-solution to $\min\limits_{\alpha\in A}\{\;-\mathcal{L}_{\alpha}v(x)\;\}=0$ in $\mathds{R}^{m}$.
Therefore uniqueness of a solution $(c_{\circ},u_{\circ}(\cdot))$ is reduced to proving that there cannot exist non-constant sub-solutions to the static HJB equation $\min\limits_{\alpha\in A}\{-\mathcal{L}_{\alpha}v \} = 0$, i.e. whether Liouville property holds for the latter static HJB. This is answered positively in \cite[Theorem 2.1]{bardi2016liouville} provided one can find a function $\psi\in C^{\infty}(\mathds{R}^m)$ and $R_{o}>0$ such that 
\begin{equation}
    \label{eq: claim proof liouv 1}
    \min\limits_{\alpha\in A}\{-\mathcal{L}_{\alpha}\psi(x)\} \geq 0 \quad \text{in }\, \overline{B(0,R_{o})}^{C},\quad \psi(x) \to +\infty\; \text{when }\, |x|\to +\infty
\end{equation}
and satisfying
\begin{equation}
\label{eq: claim proof liouv 2}
    \lim\limits_{|x|\to +\infty}\,\frac{v(x)}{\psi(x)} = 0
\end{equation}
To do so, we check that $\psi(x) \coloneqq |x|^d \log(|x|)$ satisfies the latter two conditions. Using the polynomial growth of $u_{1}$ and $u_{2}$, \eqref{eq: claim proof liouv 2} is immediate. To check the validity of \eqref{eq: claim proof liouv 1}, we compute $-\mathcal{L}_{\alpha}\psi(x)$ and make use of assumptions (A3, A4, A6). This is done in detail in \cite{kouhkouhPhD}. 
And therefore, $v=u_{1}-u_{2}\equiv \text{constant}$.
\end{proof}

\begin{remark}\label{rmk: critical value}
Recalling the definition of the corresponding dual problem \eqref{eq: dual - NL - proof}, one can see that the ergodic constant $c$ that is given by Theorem \ref{thm: main NL} is the largest one, in the sense that: if there exists another solution $(\widetilde{c}, \widetilde{u}(\cdot))$, then necessary $c\geq \widetilde{c}$. This is in line with the classical results on viscous ergodic Bellman equations for which one usually expects infinitely many possible ergodic constants (and solutions) but all smaller than the critical (largest) one; see \cite{ichihara2011recurrence, kaise2006structure}. Analogously, when the Hamiltonian is given by a $\max$ (instead of a $\min$), the ergodic constant $c$ that we obtain will be the smallest one.
\end{remark}

We conclude this section by mentioning an easy consequence of our main result and which is a continuity estimate on the (critical) ergodic constant. Such an estimate is important for applications to problems in singular perturbations and homogenization, and it is a refinement of \cite[Proposition 4.4]{barles2016unbounded}. Indeed, using the explicit definition of the ergodic constant $c$ in Theorem \ref{thm: main NL}, together with the estimate in Proposition \ref{prop: continuity F}, we can upper-bound $|c_{1}-c_{2}|$ where  $(c_{i},u_{i}(\cdot)),i=1,2,$ solve ergodic HJB equations 
\begin{equation*}
    \min\limits_{\alpha \in A_{i}} \{ -b_{i}(x,\alpha)\cdot \nabla u_{i}(x) - \text{trace}(a_{i}(x,\alpha)D^{2}u_{i}(x)) + f_{i}(x,\alpha) \} = c_{i},\quad \text{ in } \mathds{R}^{n}.
\end{equation*}

\section*{Acknowledgments}
I wish to thank Martino Bardi, J. Frédéric Bonnans, Radu Ioan Boţ, Alessandro Goffi and Boris Mordukhovich  for fruitful discussions on the content of this paper. I am also grateful to the reviewers for their valuable comments. 

\bibliographystyle{siam}
\bibliography{references}
\end{document}